\documentclass[11pt]{amsart}

\usepackage{amsthm,amssymb,amsmath,amstext,amscd,amsfonts,amsbsy,amsxtra,latexsym,hyperref,color}%,tikz,MnSymbol
\hypersetup{colorlinks=true,linkcolor=blue,citecolor=magenta}
\usepackage{fullpage}
\usepackage[english]{babel}
\usepackage[latin1]{inputenc}
\usepackage[all,cmtip]{xy}
\usepackage{comment,enumitem,bbm}
\newcommand{\field}[1]{\mathbb{#1}}
\newcommand{\N}{\field{N}}
\newcommand{\Z}{\field{Z}}
\newcommand{\R}{\field{R}}
\newcommand{\C}{\field{C}}
\newcommand{\Q}{\field{Q}}
\newcommand{\QQ}{\mathcal{Q}}
\newcommand{\HH}{\mathbb{H}}
\renewcommand{\H}{\mathbb{H}}
\newcommand\smod[1]{\ \left(\operatorname{mod} #1\right)}
\newcommand{\sgn}{\operatorname{sgn}}
\newcommand{\SL}{\operatorname{SL}}

\newcommand{\im}{\text{Im}}
\newcommand{\re}{\text{Re}}
\newcommand{\f}{\mathbbm{f}}
\makeatletter
\newcommand{\superimpose}[2]{%
  {\ooalign{$#1\@firstoftwo#2$\cr\hfil$#1\@secondoftwo#2$\hfil\cr}}}
\makeatother
\allowdisplaybreaks 

\numberwithin{equation}{section}
\newtheorem{theorem}{\textbf{Theorem}}
\newtheorem{definition}[theorem]{\textbf{Definition}}
\numberwithin{theorem}{section}
\numberwithin{table}{section}
\newtheorem{lemma}[theorem]{\textbf{Lemma}}
\newtheorem{proposition}[theorem]{\textbf{Proposition}}
\newtheorem{corollary}[theorem]{\textbf{Corollary}}
\newtheorem*{theorem*}{Theorem}
\theoremstyle{remark}
\newtheorem*{remark}{Remark}
\newtheorem*{remarks}{Remarks}

\renewenvironment{proof}[1][Proof]{\begin{trivlist}
\item[\hskip \labelsep {\bfseries #1:}]}{\qed\end{trivlist}}

\newcommand{\abs}[1]{\left\vert#1\right\vert}

\DeclareMathOperator{\Mp}{Mp}
\DeclareMathOperator{\Gr}{Gr}

\newcommand{\smallTmatrix}{\left(\begin{smallmatrix}1 & 1 \\ 0 & 1\end{smallmatrix}\right)}
\newcommand{\smallSmatrix}{\left(\begin{smallmatrix}0 & -1 \\ 1 & 0\end{smallmatrix}\right)}

\newcommand{\smallabcd}{\left(\begin{smallmatrix}a & b \\ c & d\end{smallmatrix}\right)}

\newcommand{\calL}{\mathcal{L}}
\newcommand{\calM}{\mathcal{M}}
\newcommand{\calN}{\mathcal{N}}

\newcommand{\frake}{\mathfrak{e}}

\begin{document}
\title{Central $L$-values of elliptic curves and local polynomials}
\date{\today}
\author{Stephan Ehlen}
\address{Weyertal 86-90, Universit\"at zu K\"oln, Mathematisches Institut, D-50931 K\"oln, Germany}
\email{mail@stephanehlen.de}
\author{Pavel Guerzhoy}
\address{Department of Mathematics\\
University of Hawaii\\ 
2565 McCarthy Mall \\
Honolulu, HI 96822 }
\email{pavel@math.hawaii.edu}
\author{Ben Kane} 
\address{Mathematics Department, University of Hong Kong, Pokfulam, Hong Kong}
\email{bkane@hku.hk}
\author{Larry Rolen}
\address{Department of Mathematics\\
Vanderbilt University \\
Nashville, TN 37240}
\email{larry.rolen@vanderbilt.edu}

\thanks{Part of the research was carried out while the first author was a postdoctoral fellow at the CRM and McGill University in Montreal. The research of the third author was supported by grant project numbers HKU 27300314, 17302515, 17316416, and 17301317  of the Research Grants Council of Hong Kong SAR.  Part of the research was carried out while the the fourth researcher held posdoctoral positions at the University of Cologne and the Pennsylvania State University, an assistant professorship at Trinity College, Dublin, and a visiting position at the Georgia Institute of Technology. The fourth author thanks the University of Cologne and the DFG for their generous support via the University of Cologne postdoc grant DFG Grant D-72133-G-403-151001011. }

\begin{abstract}
Here we study the recently-introduced notion of a locally harmonic Maass form and its applications to the theory of $L$-functions. In particular, we find a criterion for vanishing of certain twisted central $L$-values of a family of elliptic curves, whereby vanishing occurs precisely when the values of two finite sums over canonical binary quadratic forms coincide. This yields a vanishing criterion based on vastly simpler formulas than the formulas related to work of Birch and Swinnerton-Dyer to determine such $L$-values, and extends beyond their framework to special non-CM elliptic curves. 
\end{abstract}

\maketitle

\section{Introduction and Statement of Results}
A celebrated result in analytic number theory, due to Dirichlet, is his so-called \emph{class number formula}. This formula gives a deep relation between values of quadratic $L$-functions and class numbers of quadratic fields, which enumerate equivalence classes of binary quadratic forms. Given the historical importance of this result, it is natural to ask if similar formulas exist for other $L$-functions. A natural candidate to search for such formulas is the $L$-functions associated to rational elliptic curves. In this case, such formulas exist in special cases. For example, in a famous paper of Tunnell \cite{Tunnell}, he gave a ``strange'' formula for the $L$-values of the quadratic twists of the strong Weil curve of conductor 32, which is well-known to be related to the study of congruent numbers. Moreover, beautiful formulas involving specializations of the Weierstrass $\zeta$-function were given for such $L$-values in the case of CM elliptic curves by Birch and Swinnerton-Dyer in \cite{BSDNotes}.

Here, for rational elliptic curves for a special sequence of conductors and a natural family of fundamental discriminants $D$, we produce relatively simple formulas (taking the difference of the two sides of Theorem \ref{mainthm} below) which vanish precisely when the $L$-values of the $D$th quadratic twists of these elliptic curves vanish. These formulas mirror the classical Dirichlet class number formula in the sense that they relate (the vanishing of) central $L$-values of a set of elliptic curves to canonical sums over binary quadratic forms. Our proof is very different than those previously obtained in the literature, as it uses in an essential way the burgeoning theory of locally harmonic Maass forms defined in \cite{BKK}.

To describe this result, we first require some notation. For a discriminant $\delta$, we denote the set of binary quadratic forms of discriminant $\delta$ by $\QQ_{\delta}$. Splitting $\delta=DD_0$ for two discriminants $D_0$ and $D$, we require certain canonical sums over quadratic forms given  for $x\in \R$ by 
\[
F_{0,N,D,D_0}(x):=\sum_{\substack{Q=[a,b,c]\in \QQ_{D_0D}\\ Q(x)>0>a\\ N\mid a}}\chi_{D_0}(Q),
\]
where $\chi_{D_0}$ is the genus character, and by abuse of notation we set $Q(x)$ to be $Q(x,1)$. For $x\in\Q$, this sum is furthermore finite. Our main result is the following, where we say that $N$ is a \emph{dimension one level} if $\dim(S_2(N))=1$ and the definition of a \emph{good} discriminant for a level $N$ is given in Section \ref{sec:gooddiscs} (see Table \ref{tab:discs}).
\begin{remark}
For $x\in \Q$, one can interpret $F_{0,N,D,D_0}$ as a weighted intersection number between the union of the geodesics associated to the quadratic forms 
\[
\mathcal{Q}_{N,\Delta}:=\left\{Q=[a,b,c]\in \QQ_{\Delta}: Q(x)>0>a, N\mid a\right\}
\]
 and the vertical geodesic from $x$ to $i\infty$, weighted with the genus character $\chi_{D_0}$. We denote the elements of $\mathcal{Q}_{N,\Delta}$ intersecting the vertical line from $x$ to $i\infty$ by $\mathcal{S}_{N,\Delta}(x)$ and note that the sum with trivial weighting yields
\[
F_{0,N,DD_0,1}(x)=\#\mathcal{S}_{N,DD_0}(x).
\]
\end{remark}

\begin{theorem}\label{mainthm}
Let $N$ be a dimension one level, and suppose that $D$ is a good fundamental discriminant for $N$. Denote the elliptic curve corresponding to the unique normalized weight 2 newform of level $N$ by $E$ and let $E_D$ be its $D$th quadratic twist. Let $D_0<0$ be given as in the row for level $N$ in Table \ref{tab:discs}. Then for each good fundamental discriminant $D$ in Table \ref{tab:discs} such that $|DD_0|$ is not a square, we have $L(E_D,1)=0$ if and only if the corresponding pair $(x_{N,1},x_{N,2})$ in Table \ref{tab:discs} satisfies 
\[
F_{0,N,D,D_0}\!\left(x_{N,1}\right)= F_{0,N,D,D_0}\!\left(x_{N,2}\right).
\]
\end{theorem} 
Before discussing some arithmetically interesting corollaries that arise from Theorem \ref{mainthm}, we give a broad overview of the strategy used to prove Theorem \ref{mainthm} with some of the technical issues removed. We use the theory of theta lifts (see Section \ref{sec:theta-general}) to construct a function $\mathcal{F}_{0,N,D,D_0}$ which is a $\Gamma_0(N)$-invariant locally harmonic Maass form (roughly speaking, this means that it is invariant under the action of $\Gamma_0(N)$ and it is harmonic in certain connected components of the upper half plane $\H$, but may be discontinuous across certain geodesics). Every locally harmonic Maass form splits naturally into three components, a continuous holomorphic part, a continuous non-holomorphic part, and a locally polynomial part. In the case of invariant forms, this polynomial is moreover a constant. The function $F_{0,N,D,D_0}$ turns out to be the locally constant part of $\mathcal{F}_{0,N,D,D_0}$. The locally constant part of an invariant locally harmonic Maass form is generally only itself invariant when the other two parts vanish. By relating the continuous holomorphic and non-holomorphic parts of $\mathcal{F}_{0,N,D,D_0}$ to a special weight $2$ cusp form $f_{1,N,D,D_0}$ which was studied by Kohnen in \cite{KohnenCoeff} and shown to be related to central $L$-values, we obtain a link between the $\Gamma_0(N)$-invariance of $F_{0,N,D,D_0}$ and the vanishing of central $L$-values. We finally determine a pair of rational numbers $x_{N,1}$ and $x_{N,2}$ that are related via $\Gamma_0(N)$ such that $F_{0,N,D,D_0}$ is invariant under $\Gamma_0(N)$ if and only if it agrees at these two points. The main technical difficulties arise from the fact that the naive definitions of $\mathcal{F}_{0,N,D,D_0}$ and $f_{1,N,D,D_0}$ do not converge absolutely, requiring an analytic continuation to $s=0$ in a separate variable $s$. There are two standard approaches to construct this analytic continuation, one known as Hecke's trick and another using spectral parameters. Kohnen uses Hecke's trick in \cite{KohnenCoeff}, while theta lifts are more complementary with the spectral theory (see Section \ref{sec:theta-lifts}), which we use to construct a parallel function $\f_{1,N,D_0,D}$ via theta lifts. We are required to show that these two constructions agree at $s=0$ (see Lemma \ref{lem:f1fchi}) even though they do not agree for general $s$ because the properties of these functions differ for $s\neq 0$.
\begin{remarks}
\noindent

\noindent
\begin{enumerate}[leftmargin=*,label={\rm(\arabic*)}]
\item It would be very interesting to extend the results of Theorem \ref{mainthm} to more general conductors. However, there are several technical issues which arise. It is important to note that since we restrict to dimension one levels, the results in Theorem \ref{mainthm} only yield formulas for finitely many $N$. Similar formulas exist for more general $N$, but one must apply projection operators into the corresponding eigenspaces, yielding more complicated formulas with Hecke operators acting on the formulas. Some of the issues related to non-dimension one levels are discussed by the third author's Ph.D. student Kar Lun Kong in \cite{KongPhD}.  We concentrate in this paper on dimension one levels in order to obtain explicit formulas and then mainly concentrate on certain combinatorial applications of Theorem \ref{mainthm} for which the level $N$ happens to be a dimension one level.

\item
Using Waldspurger's Theorem in the form of the Kohnen--Zagier formula \cite{KohnenZagier}, there is an analogous result for vanishing and non-vanishing of central values for $L$-series attached to weight $k>2$ Hecke eigenforms and $N$ is squarefree and odd. In that case, each term in the sum over $Q$ defining $F_{0,N,D,D_0}$ above is multiplied by $Q(x)^{k-1}$. The main focuses of this paper are to extend this to include weight $2$ via Hecke's trick and also to consider other levels which are not necessarily squarefree and odd.
\item 
The formulas $F_{0,N,D,D_0}(x_{N,j})$ may be written rather explicitly. To demonstrate the type of formulas that one obtains, we explicitly write the formulas for $N=32$. Firstly, for all $N$ we have 
\[
F_{0,N,D,D_0}(0)=\sum_{\substack{Q=[a,b,c]\in \QQ_{D_0D}\\ c>0>a\\ N\mid a}}\chi_{D_0}(Q).
\]
In the case of $D_0=-3$, the genus character can be simplified to 
\[
\chi_{-3}([a,b,c])=\begin{cases}
\left(\frac{-3}{a}\right)&\text{ if } 3\nmid a\\
 \left(\frac{-3}{c}\right)&\text{ if }3|a.
\end{cases}
\]
For $x_{32,2}=1/3$, the sum becomes 
\[
F_{0,32,D,-3}\!\left(\frac{1}{3}\right)=\sum_{\substack{Q=[a,b,c]\in \QQ_{D_0D}\\ a+3b+9c>0>a\\ 32\mid a}}\chi_{-3}(Q).
\]
\end{enumerate}
\end{remarks}

To illustrate Theorem \ref{mainthm}, we describe explicit versions of Theorem \ref{mainthm} related to combinatorial questions. Firstly, recall that a squarefree integer $n$ is called \begin{it}congruent\end{it} if $n$ is the area of a right triangle with rational side lengths and \begin{it}noncongruent\end{it} otherwise. The case $N=32$ of Theorem \ref{mainthm} is well-known to apply to the study of congruent numbers by Tunnell's work \cite{Tunnell}.
\begin{corollary}\label{maincor}
Let $D<0$ be a fundamental discriminant $|D|\equiv 3\pmod{8}$ such that $3|D|$ is not a square, and let $E$ be the congruent number curve of conductor $32$. Then $F_{0,32,D,-3}\left(0\right)=F_{0,32,D,-3}(1/3)$ if and only if $L(E_D,1)=0$. In particular, if $|D|$ is congruent, then $F_{0,32,D,-3}\left(0\right)=F_{0,32,D,-3}(1/3)$ and the converse is true assuming the Birch and Swinnerton-Dyer conjecture (hereafter BSD).
\end{corollary}

\begin{remark}

\noindent
In \cite{Skoruppa}, Skoruppa used the theory of skew-holomorphic Jacobi forms to obtain a similar-looking condition for fundamental discriminants $D$ congruent to $1$ modulo $8$.  One can use the theory here with a slightly modified genus character to obtain a condition equivalent to Skoruppa's result. However, his method appears to have an inherent dependence on the quadratic residue of $D$ modulo $8$; a subtle generalization of his result may lead to a condition identical to that obtained in Corollary \ref{maincor}, but it does not seem to directly follow from his theory in its current form.

\end{remark}

Table \ref{tab:maincor} (constructed via GP/Pari code running for a couple of hours) illustrates the usage of Corollary \ref{maincor} in checking for congruent numbers. Note that the calculation for Corollary \ref{maincor} only proves that non-congruent numbers are not congruent (it would be an equivalence under the assumption of BSD); the congruent numbers may in practice be explicitly shown to be congruent, although this is a non-trivial exercise; see Figure 1.3 in \cite{Koblitz} for an example of Don Zagier's which shows that $157$ is congruent. 
\begin{center}
\begin{table}[th]\caption{Congruent numbers and Corollary \ref{maincor}\label{tab:maincor}}
\begin{tabular}{|c|c|c|c|}
\hline
$D$ & $F_{0,32,D,-3}(0)$ & $F_{0,32,D,-3}\left(\frac{1}{3}\right)$ & $|D|$ congruent/non-congruent \\ &&&(assuming BSD)\\
\hline
\hline
-11&0 & 1&non-congruent\\
\hline
-19&0 & 1 &non-congruent\\
\hline
-35&0 & 2 &non-congruent\\
\hline
-219& 2&2 & congruent\\
\hline
-331& 0 & 3 & non-congruent\\
\hline
-371& 4 & 4 & congruent\\
\hline
-4219& 6 & 9 & non-congruent\\
\hline
-80011 & 28&40 & non-congruent\\
\hline
-80155&24&32&non-congruent\\
\hline
-800003&138 &140 &non-congruent\\
\hline
-800011&72 &81 &non-congruent\\
\hline
-800027&86&94&non-congruent\\
\hline
-8000459&578&590&non-congruent\\
\hline
-8000467&190&200&non-congruent\\
\hline
\end{tabular}
\end{table}
\end{center}

\rm

As an amusing consequence, we next show how studying the parity of the values $F_{0,32,D,-3}$ (i.e., the parity of $\#\mathcal{S}_{32,-3D}$) may lead to results about congruent numbers. In order to illustrate this idea, we re-investigate the well-known result that every prime $p\equiv 3\pmod{8}$ is not a congruent number from this perspective. 
\begin{corollary}\label{cor:SDodd}
If $p\equiv3\pmod 8$ is a prime and $\#\mathcal{S}_{32,-3p}(1/3)$ is odd, then $L(E_p,1)\neq 0$; in particular, if this is the case, then $p$ is not a congruent number.
\end{corollary}
\begin{remarks}
\noindent

\noindent
\begin{enumerate}[leftmargin=*,label={\rm(\arabic*)}]
\item
One may explicitly compute 
\[
\mathcal{S}_{32,-3D}\!\left(\frac{1}{3}\right)=\{[a,b,c]: b^2-4ac=3D,\ 32|a,\ a+3b+9c>0>a\}.
\]
\item
Computer calculations indicate that $\#\mathcal{S}_{32,-3p}$ is always odd for $p\equiv 3\pmod{8}$ and Volker Genz has recently announced a proof that $\#\mathcal{S}_{32,-3p}$ is indeed always odd, yielding a new proof of the fact that $p$ is not a congruent number. Indeed, the calculation from Table \ref{tab:maincor} is repeated under the restriction that the discriminants are negatives of primes and the results are listed in Table \ref{tab:primes} to illustrate the parity condition hinted at in Corollary \ref{cor:SDodd}.
\begin{center}
\begin{table}[th]\caption{Primes and parity: Corollary \ref{cor:SDodd}\label{tab:primes}}
\begin{tabular}{|c|c|c|}
\hline
$p$ & $F_{0,32,-p,-3}(0)$ & $F_{0,32,-p,-3}\left(\frac{1}{3}\right)$ \\
\hline
\hline
11&0 & 1\\
\hline
19&0 & 1\\
\hline
331& 0 & 3\\
\hline
571& 4 & 1\\
\hline
5227&10&5\\
\hline
5939&18&17\\
\hline
75011&70& 83\\
\hline
75403&24&21\\
\hline
200171&102&117\\
\hline
200443&36&37\\
\hline
1300027&34&93\\
\hline
5500003&120&137\\
\hline
40500011& 1254& 1331\\
\hline
40500059& 1186& 1189\\
\hline
\end{tabular}
\end{table}
\end{center}

\end{enumerate}
\end{remarks}

We conclude by considering another example. Let $E_{27,D}$ denote the $D$th quadratic twist $E_{27,D}$ of the Fermat cubic curve $x^3+y^3=1$ of conductor $N=27$ (or its equivalent Weierstrass form $y^2=x^3-432$). Taking $N=27$ and $D_0=-4$, Theorem \ref{mainthm} yields an algorithm to determine whether or not $E_{27,D}(\Q)$ (i.e, whether the number of rational solutions to the Diophantine equation $x^3+|D|y^2=432$) is finite or infinite.
\begin{corollary}\label{cor:sumtwocubes}
For a fundamental discriminant $D<0$ with $|D|\equiv 1\pmod{3}$ and $4|D|$ not a square, we have $F_{0,27,D,-4}(0)=F_{0,27,D,-4}(1/2)$ if and only if $L(E_{27,D},1)=0$.  In particular, if $|D|\equiv 1\pmod{3}$ and $E_{27,D}$ has infinitely many rational points, then $F_{0,27,D,-4}(0)=F_{0,27,D,-4}(1/2)$ and the converse is true assuming BSD.
\rm

\end{corollary}
%\begin{remark}
%In light of Corollary \ref{cor:SDodd} and the connection with the sums of cubes in Corollary \ref{cor:sumtwocubes}, one is left to wonder whether there is a similar parity condition that might lead to the conclusion that certain primes are not the sum of two cubes. Although the case $p\equiv -1\pmod{9}$ is not contained in Corollary \ref{cor:sumtwocubes}, we note that it was conjectured by Sylvester that such $p$ are the sum of two cubes. Even a modified version of the construction used in this paper would not be beneficial in that case, however, as one requires the BSD conjecture to obtain the claim in the reverse direction.
%\end{remark}
Table \ref{tab:sum2cubes}  lists the result of the calculation analogous to Table \ref{tab:maincor}.
\begin{center}
\begin{table}[th]\caption{Data for Corollary \ref{cor:sumtwocubes}\label{tab:sum2cubes}}
\begin{tabular}{|c|c|c|c|}
\hline
$D$ & $F_{0,27,D,-4}(0)$ & $F_{0,27,D,-4}\left(\frac{1}{2}\right)$ & $E_{27,D}(\Q)$ finite/infinite\\ &&&(assuming BSD)\\
\hline
\hline
-7&0 & 2&finite\\
\hline
-31&2 & 2 &infinite\\
\hline
-115& 0&4&finite\\
\hline
-283&2 &2 &infinite\\
\hline
-3019&4&6&finite\\
\hline
-3079&24&34&finite\\
\hline
-3115&8&8&infinite\\
\hline
-30091&44&26&finite\\
\hline
-30139&20&14&finite\\
\hline
-600004&158&196&finite\\
\hline
-600007&132&132&infinite\\
\hline
-600019&130&172&finite\\
\hline
-600043&70 &58 &finite\\
\hline
-3000004&368&336&finite\\
\hline
-3000103&432&414&finite\\
\hline
-3000115&224&224&infinite\\
\hline
\end{tabular}
\end{table}
\end{center}
The paper is organized as follows. In Section \ref{sec:prelim}, we establish the basic definitions needed for the paper, including genus character, good discriminants, locally harmonic Maass forms, and Poincar\'e series, review the work of Kohnen, and give a few computational examples related to the dimension one levels. In Section \ref{sec:theta-general}, we describe a generalization of Kohnen's work using theta lifts, obtaining certain locally harmonic Maass forms as theta lifts; these locally harmonic Maass forms converge to the sums $F_{0,N,D,D_0}$ as one approaches a cusp. In Section \ref{sec:mainproofs}, we give the proofs of Theorem \ref{mainthm} and Corollaries \ref{maincor}, \ref{cor:SDodd}, and \ref{cor:sumtwocubes}. 

\section*{Acknowledgements}
\noindent 
The authors thank Kathrin Bringmann for large contributions to this paper which she generously shared with us. The authors are also grateful to Ken Ono for enlightening conversations concerning the paper, and to Volker Genz for sharing his proof of the fact that $\#\mathcal{S}_{32,-3p}$ is odd. The authors thank Hongbo Yin for pointing out an error in an earlier version of this paper. The authors also thank the anonymous referee for useful comments that helped with the exposition of the paper. 

\section{Preliminaries}\label{sec:prelim}

\subsection{Definitions}\label{Defns}
Let $D_0$ and $D$ be fixed fundamental discriminants such that $DD_0>0$ is a discriminant. We begin by introducing the genus character $\chi_{D_0}$ on quadratic forms $Q=[a,b,c]\in \QQ_{D_0D}$, which is defined as
\begin{equation}\label{eqn:genusdef}
\chi_{D_0}\left(Q\right):=\begin{cases} 
\left(\frac{D_0}{r}\right) & \text{if }\left(a,b,c, D_0\right)=1\text{ and $Q$ represents $r$ with $\left(r,D_0\right)=1$,}\\ 0 & \text{if }\left(a,b,c,D_0\right)>1.\end{cases}
\end{equation}
The genus character is independent of the choice of $r$ and only depends on the $\SL_2(\Z)$-equivalence class of $Q$ (actually, it only depends on the genus of $Q$ and is indeed a character on the abelian group formed by the elements of the class group with multiplication given by Gauss's composition law after factoring out by squares). We then define for $x\in\Q$ the sum
\begin{equation}\label{eqn:F1}
F_{0,N,D,D_0}(x)=\sum_{\substack{Q=[a,b,c]\in \QQ_{D_0D}\\ Q(x)>0>a\\ N\mid a}}\chi_{D_0}(Q).
\end{equation}

\begin{remark}
\noindent
We note that for rational $x$, this sum is actually finite. In fact, if $x=p/q$ then by \cite{ZagierQuadratic}, we have 
\[
D_0Dq^2=\left|bq+2ap\right|^2+4\left|a\right|\left|ap^2+bpq+cq^2\right|
\]
(note that this corrects a typo in \cite{ZagierQuadratic}). This bounds each of $a,b$, and $c$. 
\end{remark}

\subsection{Explicit computations and good discriminants}\label{sec:gooddiscs}
We call an odd discriminant $D$ \begin{it}good\end{it} for level $N$ if the following hold:
\noindent

\noindent
\begin{enumerate}[leftmargin=*,label={\rm(\arabic*)}]
\item If $N$ is not a perfect square, then $\left(\frac{-4N}{|D|}\right)=1$.
\item If $2\mid N$, then $|D|\equiv 3\pmod{8}$.
\item If $p\equiv -1\pmod{8}$ is a prime with $p| N$, then $\left(\frac{-p}{|D|}\right)=-1$.
\item If $p\equiv 3\pmod{8}$ is a prime with $p^r\| N$ for some $r\geq 1$, then $\left(\frac{-p}{|D|}\right)=(-1)^{r+1}$.
\end{enumerate}
We next compile a list of data for each of the dimension one levels; specifically, in Table \ref{tab:discs} we choose one fundamental discriminant $D_0$ and compute the good fundamental discriminants $D$. We then give a brief list of the first few discriminants for which the sums in \eqref{eqn:F1} are not invariant under $\Gamma_0(N)$, underlining the good fundamental discriminants for which we conclude that the central twist $L(E_D,1)$ does not vanish by Theorem \ref{mainthm}. Finally we give a pair of choices of $x_{N,1}$ and $x_{N,2}$ used to describe Theorem \ref{mainthm}.
\begin{center}
\begin{table}[th]\caption{Table of good fundamental discriminants\label{tab:discs}}
\begin{tabular}{|c|c|c|c|c|}
\hline
Level $N$ &$D_0$ & Good Discriminants $D$ & Brief list of $m=|D|$ with \eqref{eqn:F1} not invariant &  $(x_{N,1},x_{N,2})$\\
\hline
\hline
11& $-3$ & $\left(\frac{-11}{|D|}\right)= 1$& $4,11,12,\underline{15},16,20,\underline{23},27,\underline{31},44,48,\dots$ & $0,\frac{1}{3}$\\
\hline
14& $-3$ & $\left(\frac{-4\cdot 14}{|D|}\right)=1$& $\underline{19},20,24,27,35,40,52,56,\underline{59},68,\dots$ & $0,\frac{1}{2}$\\
\hline
15 &$-4$ &$\substack{\left(\frac{5}{|D|}\right)=1\\ \text{and }\left(\frac{-3}{|D|}\right)\neq -1} $ &$15,16,\underline{19},24,\underline{31},\underline{39},40,\underline{51},55,60,\dots$ &$0,\frac{1}{3}$\\
\hline
17&$-7$ & $\left(\frac{-4\cdot 17}{|D|}\right)=1$  &$\underline{3},\underline{11},20,\underline{23},24,28,\underline{31},40,48,51,63,\dots$ &$0,\frac{1}{2}$\\
\hline
19 &$-4$& $\left(\frac{-19}{|D|}\right)= 1$& $\underline{7},\underline{11},19,\underline{20},\underline{24},28,\underline{35},36,\underline{39},43,44,\dots$  &$0,\frac{1}{2}$\\
\hline
20& $-3$ & $\substack{|D|\equiv 3\pmod{8}\\ \text{ and }\left(\frac{-20}{|D|}\right)= 1}$& $27,35,\underline{43},\underline{67},\underline{83},\underline{107},115,\underline{123}, \dots$ & $0,\frac{1}{2}$\\
\hline
21& $-19$& $\substack{\left(\frac{-7}{|D|}\right)=-1\\ \text{ and }\left(\frac{-3}{|D|}\right)=1}$ &$\underline{3},7,\underline{24},27,28,\underline{31},\underline{40},48,\underline{52},63,\dots$ &$0,\frac{1}{2}$\\
\hline
24 &$-11$ &$\substack{|D|\equiv 3\pmod{8}\\ \text{and }\left(\frac{-24}{|D|}\right)=1}$ &$\underline{3},27, \underline{35},\underline{51},\underline{59},75,\underline{83},99,\underline{107},\underline{123},\dots$ & $\frac{1}{2},\frac{1}{3}$\\
\hline
27 & $-4$ &$\left(\frac{-3}{|D|}\right)=1$ &$\underline{7},\underline{19},28,36,\underline{40},\underline{43},\underline{52},\underline{55},64,\underline{67},76,\dots$ & $0,\frac{1}{2}$\\
\hline
32 & $-3$& $|D|\equiv 3\pmod{8}$ & $\underline{11},12,\underline{19},\underline{35},\underline{43},48,\underline{51},\underline{59},\underline{67},75,\underline{83}\dots$ &$0,\frac{1}{3}$\\
\hline
36&$-11$ & $\substack{|D|\equiv 3\pmod{8}\\ \text{and }\left(\frac{-3}{|D|}\right)=-1}$ &$27,\underline{35},\underline{59},\underline{83},99,\underline{107},\underline{131},\underline{155},171,\dots$ &$0,\frac{1}{2}$\\
\hline
49&$-3$ & $\left(\frac{-7}{|D|}\right)=-1$ &$\underline{19},\underline{20},27,\underline{31},\underline{40},\underline{47},48,\underline{55},\underline{59},\underline{68},75,\dots$ &$0,\frac{1}{7}$\\
\hline
\end{tabular}
\end{table}
\end{center}

\subsection{The Work of Kohnen}
In this section we review the main results of \cite{KohnenCoeff} which we need. We first recall a special set of cusp forms of weight $2k$. For integers $k\geq2$, $N\geq1$, and fundamental discriminants $D,D_0$ with $DD_0>0$, let 
\[
f_{k,N,D,D_0}(\tau):=\sum_{\buildrel{[a,b,c]\in\mathcal{Q}_{DD_0}}\over{N|a}}\chi_{D_0}(a,b,c)\left(a\tau^2+b\tau+c\right)^{-k}.
\]
This series converges absolutely on compact sets and defines a cusp form of weight $2k$ on $\Gamma_0(N)$. We require the corresponding cusp form in the case of $k=1$, where the above series is not absolutely convergent. In this case, we define by the ``Hecke trick'' 
\[
f_{1,N,D,D_0}(\tau;s):=\sum_{\buildrel{[a,b,c]\in\mathcal{Q}_{DD_0}}\over{N|a}}\chi_{D_0}(a,b,c)\left(a\tau^2+b\tau+c\right)^{-k}\left|a\tau^2+b\tau+c\right|^{-s}.
\]
This has an analytic continuation to $s=0$ as Kohnen showed in the proof of Proposition 2 of \cite{KohnenCoeff}, and we set $f_{1,N,D,D_0}(\tau):=\left[f_{1,N,D,D_0}(\tau;s)\right]_{s=0}\in M_2(N)$, where $[g]_{s=s_0}$ denotes the analytic continuation of $g$ to $s=s_0$ and $M_{2k}(N)$ (resp. $S_{2k}(N)$) denotes the space of weight $2k$ holomorphic modular forms (resp. cusp forms) on $\Gamma_0(N)$.  We note that in general, this is \emph{not} a cusp form. We next describe how one uses these functions to obtain information about central critical $L$-values. Proposition 7 of \cite{KohnenCoeff} connects the inner product of cusp forms with this family of cusp forms and cycle integrals. To describe this, first define
\[
r_{k,N}\left(g;D_0,\left|D\right|\right):=\sum_{\substack{Q=[a,b,c]\in \Gamma_0(N)\backslash\QQ_{DD_0}\\ N\mid a }} \chi_{D_0}(Q)\int_{C_Q}f(z) d_{Q,k}z.
\]
where $C_Q$ is the image in $\Gamma_0(N)\backslash\mathbb{H}$ of the semicircle $a|\tau|^2+b\operatorname{Re}(\tau)+c=0$ (oriented from left to right if $a>0$, from right to left if $a<0$, and from $-\frac{c}{b}$ to $i\infty$ if $a=0$) and $d_{Q,k}z:=\left(az^2+bz+c\right)^{k-1}dz.$  Although Kohnen assumed that $N$ was squarefree and odd in \cite{KohnenCoeff}, this condition was not used in Proposition 7 of \cite{KohnenCoeff}, which we state here for general $N$.
\begin{theorem}[Kohnen \cite{KohnenCoeff} Proposition 7]
Let $f\in S_{2k}(N)$. Then we have that
\[
\langle f,f_{k,N,D,D_0}\rangle=[\Gamma_0(N)\colon\operatorname{SL}_2(\Z)]^{-1}\pi\binom{2k-2}{k-1}2^{-2k+2}\left(|DD_0|\right)^{1/2-k}r_{k,N,D,D_0}(f).
\]
\end{theorem}

In our situation, we require forms of more general level than those considered in Kohnen. For example, in the congruent number case when $k=1$ and $N=32$, there are technical difficulties both because of convergence for the $k=1$ case and some results of Kohnen only holding for $N$ squarefree and odd.  However, we temporarily ignore these difficulties and motivate the definition of $f_{k,N,D,D_0}$ with the following theorem of Kohnen for $N$ squarefree and odd.
\begin{theorem}[Kohnen \cite{KohnenCoeff} Corollary 3]\label{thm:Kohnen}
Let $f\in S_{2k}(N)$ and let  $D$ and $D_0$ be fundamental discriminants with $(-1)^kD,(-1)^kD_0>0$, and $\left(\frac{D}{\ell}\right)=\left(\frac{D_0}{\ell}\right)=w_{\ell}$ for all primes $\ell|N$, where $w_{\ell}$ is the eigenvalue of $f$ under the Atkin-Lehner involution $W_{\ell}$. Then
\[\left(DD_0\right)^{k-1/2}L(f\otimes\chi_D,k)L(f\otimes\chi_{D_0},k)=\frac{(2\pi)^{2k}}{(k-1)!^2}2^{-2\nu(N)}|r_{k,N,D,D_0}(f)|^2.\]
\end{theorem}
Thus, if we pick a $D_0$ with $L(f\otimes\chi_{D_0},1)\neq 0$, then (again with $N$ squarefree and odd)
\[
\langle f,f_{k,N,D,D_0}\rangle\doteq L(f\otimes\chi_D,1)L(f\otimes\chi_{D_0},1)
\]
vanishes if and only if $L(f\otimes\chi_{D},1)=0.$  Here $\doteq$ means that the identity is true up to a non-zero constant.

We now discuss the basic argument for the proof of a condition analogous to that given in Theorem \ref{mainthm} for $N$ squarefree and odd whenever the space of cusp forms is one-dimensional (this condition can be relaxed by using the Hecke operators).  One would construct a locally harmonic Maass form with special properties whenever $f_{k,N,D,D_0}=0$.  Namely, the locally harmonic Maass form is a local polynomial (of degree at most $2k-2$) if and only if $f_{k,N,D,D_0}=0$.  Taking the limit of the resulting local polynomials towards cusps yields conditions resembling Theorem \ref{mainthm}.  

We now return to the complications arising from the fact that $N$ is not necessarily squarefree and odd, but we assume that $S_{2k}(N)$ is one-dimensional.  In this case, we use Waldspurger's Theorem \cite{Waldspurger} directly to obtain the connection to central critical $L$-values.  We say that $m>0$ is \begin{it}admissible for $D_0$\end{it} if $(-1)^km$ is a fundamental discriminant and $\frac{(-1)^km}{D_0}\in \Q_p^{\times^2}$ for every prime $p\mid N$.  In particular, the good fundamental discriminants in Table \ref{tab:discs} satisfy these conditions for $k=1$ and the given level $N$, yielding admissible $m>0$ for each of the listed fundamental discriminants $D_0$ in the cases considered.
\begin{lemma}\label{lem:Lval}
Let a fundamental discriminant $D_0$ be given and suppose that $f\in S_{2k}(N)$ generates  $S_{2k}(N)$. Suppose that there exists an admissible $m_0$ for $D_0$ such that that $L\left(f\otimes \chi_{(-1)^km_0},k\right)\neq 0$ and the cuspidal part of $f_{k,N,(-1)^km_0,D_0}$ is non-zero.  Then for every $m>0$ admissible for $D_0$, we have 
$$
\left<f,f_{k,N,(-1)^km,D_0}\right>=0
$$
if and only if 
$$
L\left(f\otimes \chi_{(-1)^km},k\right)=0.
$$
\end{lemma}
\begin{remark}
Admissible $m_0$ satisfying the conditions given in Lemma \ref{lem:Lval} may be found in practice by computing a few coefficients of the form $f_{k,N,(-1)^km_0,D_0}$ and using the valence formula. For $k=1$ we need to project into the subspace of cusp forms, while for $k>1$ we only need to check one coefficient because $S_{2k}(N)$ is one-dimensional. For the choices of $D_0$ given in Table \ref{tab:discs}, $m_0$ may be obtained by choosing from the list of non-vanishing choices computed in the same row of the table. For example, for $k=1$, $N=32$, and $D_0=-3$, one can check that $m=11$ is admissible and satisfies the required conditions. Moreover, potential choices of $m_0$ may be obtained by finding associated $D$ for which $F_{0,N,D,D_0}(x)$ defined in \eqref{eqn:F1} does not satisfy modularity for one fundamental discriminant $D$.   This is a finite calculation, and one expects that this holds for any fundamental discriminant $D$ for which the central $L$-value does not vanish.  If $k>1$, then to verify the above condition one multiplies each summand of $F_{0,N,D,D_0}(x)$ by $Q(x,1)^{k-1}$ and again checks modularity.
\end{remark}
\begin{proof}
Let $g_{k,N,D,D_0}$ be the cuspidal part of $f_{k,N,D,D_0}$.  Consider the generating function of the $f_{k,N,D_0,D}$ (defined in (3) of \cite{KohnenCoeff})
\begin{multline}\label{eqn:Omegadef}
\Omega_{N,k}\left(z,\tau;D_0\right):=\frac{\left[\SL_2(\Z):\Gamma_0(N)\right]2\sqrt{D}}{\pi \binom{2k-2}{k-1}}\\
\times\sum_{\substack{m\geq 1\\ (-1)^km\equiv 0,1\pmod{4}}} \sqrt{m}\sum_{t\mid N}\mu(t)\chi_{D_0}(t)f_{k,\frac{N}{t},(-1)^km,D_0}\left(z\right)e^{2\pi i m\tau}.
\end{multline}
By Theorem 1 of \cite{KohnenCoeff} (which holds for general $N$), $\Omega_{N,k}$ is a weight $k+\frac{1}{2}$ cusp form as a function of $\tau$.  

Since the space of cusp forms is one-dimensional, $S_2\left(\frac{N}{t}\right)$ is trivial for all $t>1$, and hence $g_{k,\frac{N}{t},(-1)^km,D_0}=0$ for all $t>1$.  Using the fact that the Eisenstein series are orthogonal to cusp forms, for $f\in S_{2k}(N)$ we hence have 
\begin{multline*}
G_f(\tau):=\left<f,\Omega_{N,k}\left(\cdot, \tau;D_0\right)\right>=\frac{\left[\SL_2(\Z):\Gamma_0(N)\right]2\sqrt{D}}{\pi \binom{2k-2}{k-1}}\\
\times\sum_{\substack{m\geq 1\\ (-1)^km\equiv 0,1\pmod{4}}} \sqrt{m}\left<f,g_{k,N,(-1)^km,D_0}\right>e^{2\pi i m\tau}.
\end{multline*}
Therefore, $\sqrt{m}\left<f,g_{k,N,(-1)^km,D_0}\right>$ is the $m$th coefficient of the weight $k+\frac{1}{2}$ cusp form $G_f$.  Using an argument of Parson \cite{Parson}, one can show that the action of the integral-weight Hecke operators on $\Omega_{N,k}$ in the $z$ variable equal the action of the half-integral weight operators on $\Omega_{N,k}$ in the $\tau$ variable.  Hence if $f$ is a newform, we obtain that $G_f$ is an eigenform with the same eigenvalues as $f$.  Thus, by Waldspurger's Theorem \cite{Waldspurger}, if $m$ is admissible for $D_0$, then the ratio of the $m$th coefficient and the $m_0$th coefficient of $G_f$ is proportional to the ratio of the central $L$-values $L\left(f\otimes \chi_{(-1)^km},k\right)$ and $L\left(f\otimes \chi_{(-1)^km_0},k\right)$.  Since $S_{2k}(N)$ is one-dimensional, there exists a constant $a_m$ for which 
$$
g_{k,N,(-1)^km,D_0}=a_mf.
$$
By our choice of $m_0$, we have $a_{m_0}\neq 0$ and hence 
$$
\left<f,g_{k,N,(-1)^km_0,D_0}\right>=a_{m_0}\|f\|^2\neq 0.
$$
Since $L\left(f\otimes \chi_{(-1)^km_0},k\right)\neq 0$, we obtain that
$$
\left<f,f_{k,N,(-1)^km,D_0}\right>=\left<f,g_{k,N,(-1)^km,D_0}\right>=0
$$
if and only if 
$$
L\left(f\otimes \chi_{(-1)^km},k\right)=0.
$$
\end{proof}
\subsection{Weak Maass forms}
We begin by defining the notion of a weak Maass form, a comprehensive survey of which can be found in, e.g., \cite{OnoCDM}.  Maass forms were introduced by Maass and generalized by Bruinier and Funke in \cite{BruinierFunke} to allow growth at the cusps. Following their work, we define a weak Maass form of weight $\kappa\in \frac12 \Z$ for a congruence subgroup $\Gamma$ as follows.  We first recall the usual weight $\kappa$ hyperbolic Laplacian operator given by 
\begin{equation*}
\Delta_{\kappa}:= -v^2\left( \frac{\partial^2}{\partial u^2}+\frac{\partial^2}{\partial v^2} \right) -i\kappa v\left(  \frac{\partial}{\partial u}  +\frac{\partial}{\partial v} \right).
\end{equation*}
For half-integral $\kappa$, we require that $\Gamma$ has level divisible by $4$. For this case we also define $\varepsilon_d$ for odd $d$ by
\begin{equation*}
\varepsilon_d:= \begin{cases}1 &\textrm{ if }d\equiv 1\pmod{4},\\ i&\textrm{ if }d\equiv 3\pmod{4}.\end{cases}
\end{equation*}

Let $s\in\C$ be given and define $\lambda_{\kappa,s}:=\left(s-\frac{\kappa}{2}\right)\left(1-s-\frac{\kappa}{2}\right)$ (note the symmetries $\lambda_{2-\kappa,s}=\lambda_{\kappa,s}=\lambda_{\kappa,1-s}$). We may now define weak Maass forms as follows. 
\begin{definition}\label{MaassDefn}
A \begin{it}weak Maass form of weight $\kappa$ and eigenvalue $\lambda_{\kappa,s}$ on a congruence subgroup $\Gamma\subseteq\operatorname{SL}_2(\Z)$\end{it} (with level $4|N$ if $\kappa\in\frac12+\Z$) is any $\mathcal C^2$ function $F\colon \mathbb H\to \C$ satisfying:
\begin{enumerate}[leftmargin=*,label={\rm(\arabic*)}]
\item For all $\gamma\in \Gamma,$ 
\[
F(\gamma \tau) = \begin{cases} (c\tau+d)^{\kappa}F(\tau)& \mathrm{if } ~\kappa\in \Z,\\ \nu_{\theta}^k(\gamma) F(\tau) & \mathrm{if }~ \kappa\in \frac 12+\Z.
 \end{cases}
\]
where $\nu_{\theta}(\gamma)$ is the multiplier of the theta function $\Theta(\tau):=\sum_{n\in\Z}q^{n^2}$ with $q:=e^{2\pi i \tau}$.
\item  We have that 
\[
\Delta_{\kappa}(F)=\lambda_{\kappa,s} F.
\]
\item 
As $v\rightarrow\infty$, there exist $a_1,\ldots, a_N\in\C$ such that 
\[
F(\tau)-\sum_{m=1}^Na_m\mathcal M_{\kappa,s}\left(4\pi\mathrm{sgn}(\kappa)mv\right)e^{2\pi im \sgn(\kappa)u}
\]
grows at most polynomially in $v$.  Analogous conditions are required at all cusps.
\end{enumerate}
\end{definition}
Here 
\begin{equation}\label{eqn:Mdef}
\mathcal{M}_{\kappa,s}(t):=|t|^{-\frac{\kappa}{2}}M_{\frac{\kappa}{2}\sgn(t),s-\frac{1}{2}}(|t|),
\end{equation}
where $M_{s,t}$ is the usual $M$-Whittaker function. In the case when the eigenvalue of the weak Maass form is zero (i.e., if $s=\frac{\kappa}{2}$ or $s=1-\frac{\kappa}{2}$), we call the object a \begin{it}harmonic weak Maass form,\end{it} and we denote the space of harmonic weak Maass forms of weight $\kappa$ on $\Gamma_0(N)$ by $H_{\kappa}(N)$. For $2k\in 2\N$, there are two canonical operators mapping from harmonic weak Maass forms of weight $2-2k$ to classical modular forms of weight $2k$.  These are defined by
\[
\xi_{2-2k,\tau}:=\xi_{2-2k}:=2i v^{2-2k}\overline{\frac{\partial}{\partial\overline{\tau}}},\]
\[
\mathcal{D}^{2k-1}:=\left(\frac{1}{2\pi i} \frac{\partial}{\partial\tau}\right)^{2k-1},
\]
and these operators act by \cite{OnoCDM}
\[
\xi_{2-2k}\colon H_{2-2k}(N)\rightarrow M_{2k}(N),
\]
\[
\mathcal{D}^{2k-1}\colon H_{2-2k}(N)\rightarrow M_{2k}^!(N).
\]
It is interesting to note here that in fact the image of $\mathcal{D}^{2k-1}$ is the space orthogonal to cusp forms under the (regularized) Petersson inner product, which we see in the next subsection is very different from the local Maass form situation.

\subsection{Poincar\'e series}
In this subsection, we define several types of Poincar\'e series which provide useful bases for spaces of weak Maass forms and are used later in the paper to provide explicit examples of local Maass forms using theta lifts.

Denote the $D$th weight $\frac{3}{2}-k$ Maass--Poincar\'e series at $i\infty$ for $\Gamma_0(4N)$ with eigenvalue $\lambda_{\frac{3}{2}-k,s}=\left(s-\frac{k}{2}-\frac{1}{4}\right)\left(\frac{3}{4}-\frac{k}{2}-s\right)$ by $P_{\frac{3}{2}-k,D,s}$.  For $\re(s)>1$ these are given by
\begin{equation}\label{eqn:Poincdef}
P_{\frac{3}{2}-k,D,s}(\tau) := \sum_{\gamma\in \Gamma_{\infty}\backslash \Gamma_0(4N)} \psi_{-D,\frac{3}{2}-k}(s;\tau)\Big|_{\frac{3}{2}-k} \gamma\Big| \operatorname{pr},
\end{equation}
where $\operatorname{pr}$ is Kohnen's projection operator (cf. p. 250 of \cite{KohnenCoeff}) into the plus space and for $\kappa\in \frac{1}{2}\Z$ and $m\in \Z$ we define
\begin{equation}\label{eqn:psidef}
\psi_{m,\kappa}(s;\tau):=\left(4\pi |m|\right)^{\frac{\kappa}{2}}\Gamma(2s)^{-1}\mathcal{M}_{\kappa,s}(4\pi m v)e^{2\pi i mu}.
\end{equation}
Here the plus space of weight $\kappa+\frac{1}{2}$ is the subspace of forms for which the $n$th coefficient vanishes unless $(-1)^{\kappa} n\equiv 0,1\pmod{4}$. Moreover, for $\frac{3}{4}\leq \re\left(s_0\right)\leq 1$, one defines $P_{\frac{3}{2}-k,D,s_0}(\tau):=\left[P_{\frac{3}{2}-k,D,s}(\tau)\right]_{s=s_0}$.    The Poincar\'e series $P_{\frac{3}{2}-k,D,s}$ are weak Maass forms with eigenvalue $\lambda_{\frac{3}{2}-k,s}$.  In particular, for $s=\frac{k}{2}+\frac{1}{4}$ one obtains the harmonic weak Maass forms 
$$
P_{\frac{3}{2}-k,D}:=P_{\frac{3}{2}-k,D,\frac{3}{4}}.
$$
(see Theorem 3.1 of \cite{BringmannOno}).  

For $\re(s)>1$, we also denote the $D$th Maass-Poincar\'e series for $\Gamma_0(4N)$ at $\infty$ of weight $k+\frac{1}{2}$ by 
$$
P_{k+\frac{1}{2},D,s}(\tau):=\sum_{\gamma\in \Gamma_{\infty}\backslash \Gamma_0(4N)} \psi_{D,k+\frac{1}{2}}(s;\tau)\Big|_{\frac{3}{2}-k} \gamma\Big| \operatorname{pr}
$$
and its analytic continuation to $s=\frac{k}{2}+\frac{1}{4}$ by $P_{k+\frac{1}{2},D}$.  The Poincar\'e series $P_{k+\frac{1}{2},|D|}$ are the classical weight $k+\frac{1}{2}$ cuspidal Poincar\'e series.  The definition for integral weight Poincar\'e series closely resemble the definition of \eqref{eqn:Poincdef}, but we do not require these for our purposes here.

\subsection{Locally harmonic Maass forms}
In this section, we review some basic facts pertaining to and operators acting on locally harmonic Maass forms from \cite{BKK,BKM,BKS}. We note that the results here are formally true in the case $k=1$ as well once we have solved the convergence issues with the proofs adjusted mutatis mutandis. The main objects of this paper are locally harmonic Maass forms, which, as the name suggests, are closely related to weak Maass forms.  These are a special case of a more general object, which we now define given a measure zero set $E$ (the ``exceptional set'').
\begin{definition}\label{LocalMaassDefn} A local Maass form with exceptional set $E$, weight $\kappa\in 2\Z$, and eigenvalue $\lambda_{\kappa,s}$ on a congruence subgroup $\Gamma\subseteq\operatorname{SL}_2(\Z)$ is any function $\mathcal F\colon \mathbb H\to \C$ satisfying:
\begin{enumerate}[leftmargin=*,label={\rm(\arabic*)}]
\item For all $\gamma\in \Gamma,$ 
\[
\mathcal F\vert_{\kappa}\gamma=\mathcal F.
\]
\item  For every $\tau\not\in E$, there exists a neighborhood around $\tau$ on which $\mathcal F$ is real-analytic and  
\[
\Delta_{\kappa} \mathcal F=\lambda_{\kappa,s} \mathcal F.
\]
\item For $\tau\in E$ we have
\[
\mathcal F(\tau)=\frac 12\lim_{r\rightarrow0^+}\left(\mathcal F(\tau+ir)+\mathcal F(\tau-ir)\right).
\]
\item We have that $\mathcal F$ has at most polynomial growth at the cusps.
\end{enumerate}
If $\mathcal{F}$ is a local Maass form of eigenvalue $0$, then we call $\mathcal{F}$ a \begin{it}locally harmonic Maass form.\end{it}
\end{definition}

Note that the last condition on polynomial growth is impossible for harmonic weak Maass forms to satisfy (except for classical modular forms), so that although we lose continuity, we gain nicer growth conditions. These functions first arose as natural lifts under the $\xi$-operator of $f_{k,D}:=f_{k,1,D,1}$, which themselves are important in the theory of modular forms with rational periods, along with the theory of Shimura and Shintani lifts. We are particularly interested in the case when the exceptional set is formed by a special set of geodesics corresponding to a fundamental discriminant $D$.
\[
E_D:=\left\{\tau=u+i v\in\H : \exists a,b,c\in\Z,\ b^2-4ac=D,\ a|z|^2+bx+c=0\right\},
\]

We now give a summary of the important examples and maps between locally harmonic Maass forms proved in \cite{BKK}. They define for each fundamental discriminant $D>0$ a local Maass form $\mathcal F_{1-k,D}$ which has exceptional set $E_D$. Moreover, it has the following remarkable relation to Kohnen's functions $f_{k,D}$, for some non-zero constants $\alpha_D,\beta_D$:
\begin{equation}\label{eqn:FxiD}
\xymatrix{
\mathcal{F}_{1-k,D}\ar@/_/@{->}[rr]_{\frac{1}{\beta_D}\mathcal{D}^{2k-1} }\ar@/^/@{->}[rr]^{\frac{1}{\alpha_D}\xi_{2-2k}} &&f_{k,D},
}
\end{equation}
which we recall from above is an impossible property for weak Maass forms to satisfy. Using these facts, there is a useful decomposition for the local Maass form $\mathcal F_{1-k,D}$ in terms of the Eichler integrals, defined for $f\in S_{2k}(N)$ by
\begin{align}
\label{eqn:Eichnonhol}
f^*(\tau)&:=(2i)^{1-2k}\int_{-\overline{\tau}}^{i\infty}f^{c}(z)(z+\tau)^{2k-2} dz,\\
\label{eqn:Eichhol}\mathcal E_f(\tau)&:=\sum_{n\geq1}\frac{a_{f}(n)}{n^{2k-1}}q^n.
\end{align}
Here $f^c(z):=\overline{f\left(-\overline{z}\right)}$ is the cusp form whose Fourier coefficients are the complex conjugates of the coefficients of $f$.  Recall that there exist non-zero constants $c_1$ and $c_2$ such that 
\begin{align*}
\xi_{2-2k}\left(f_{k,D}^*(\tau)\right) &= c_1f_{k,D}(\tau), & \mathcal{D}^{2k-1}\left(f_{k,D}^*(\tau)\right) &= 0,\\
\xi_{2-2k}\left(\mathcal{E}_{f_{k,D}}(\tau)\right) &= 0, & \mathcal{D}^{2k-1}\left(\mathcal{E}_{f_{k,D}}(\tau)\right) &= c_2f_{k,D}(\tau).
\end{align*}
It is then not difficult to show that there exist local polynomials $P_D$ of degree at most $2k-2$ ($P_D$ equals a fixed polynomial on each connected component of $\H\backslash E_D$) such that
\begin{equation}
\label{decomplocalmaass}
\mathcal{F}_{1-k,D}=P_D+\frac{\alpha_D}{c_1}f_{k,D}^*+\frac{\beta_D}{c_2}\mathcal E_{f_{k,D}},
\end{equation}
where $\alpha_D$ and $\beta_D$ are the constants in \eqref{eqn:FxiD}.  A similar decomposition for $k=1$ was already remarked in \cite{BKK} as having been studied by H\"ovel in \cite{Hoevel}, and follows directly from the decomposition \eqref{eqn:FxiD}.  Indeed, one sees that $P_D(\tau):=\mathcal{F}_{1-k,D}(\tau) - \frac{\alpha}{c_1}f_{k,D}^*(\tau)-\frac{\beta}{c_2}\mathcal{E}_{f_{k,D}}(\tau)$ is annihilated by both $\xi_{2-2k}$ and $\mathcal{D}^{2k-1}$ for $\tau\notin E_D$.  Since $\xi_{2-2k}\left(P_D\right)=0$, $P_D$ is locally holomorphic, while the only holomorphic functions annihilated by $\mathcal{D}^{2k-1}$ are polynomials of degree at most $2k-2$.  It turns out that the decomposition \eqref{decomplocalmaass} plays a key role in our proof of Theorem 1.1 later. 

\section{Theta lifts}\label{sec:theta-general}

In this section we will define two theta lifts and study their relation. 
We remark that variants of both theta functions can be found in the literature. 
We give some references below but also provide the explicit construction of these functions for the convenience of the reader.

We recall the basic setup for vector-valued theta functions, following the exposition of Borcherds \cite{boautgra}.
Let $L$ be an even lattice with quadratic form $Q$ of type $(b^+, b^-)$ and let $\calL=L'/L$ be the associated finite quadratic module
with the reduction of $Q$ modulo $\Z$ as quadratic form.  We write $(\cdot,\cdot)$ for the associated bilinear form so that $2Q(x)=(x,x)$.
We let $p$ be a harmonic polynomial on $\R^{b^+,b^-}$ (with respect to the Laplacian on $\R^{b^+ + b^-}$), homogeneous of degree $m^+$ on $\R^{b^+}$ and of degree $m^-$ on $\R^{b^-}$. 
If $\alpha$ is an isometry from $L \otimes \R$ to $\R^{b^+,b^-}$, then the preimage of $\R^{b^+}$ defines a point $z = z(\alpha)$ in the Grassmannian $\Gr(L)$ of $b^+$-dimensional positive definite subspaces of $L \otimes \R$. 
Similarly, the preimage of $\R^{b^-}$ is a $b^-$ dimensional negative definite subspace of $L \otimes \R$, equal to $z^\perp$. 
We write $\lambda_z$ for the orthogonal projection to $z$ and $\lambda_z^\perp$ for the one to $z^\perp$ such that
$\lambda = \lambda_z + \lambda_{z^\perp}$ for all $\lambda \in L \otimes \R$. If $z = z(\alpha)$, we sometimes also write $\lambda_\alpha$ and $\lambda_{\alpha^\perp}$.

Recall that there is a unitary representation $\rho_\calL$ of $\Mp_2(\Z)$ acting on the group ring $\C[\calL]$, called the Weil representation associated with $\calL$. For details we refer to \cite{boautgra}.
We write $\frake_\mu$ for the standard basis element of $\C[\calL]$ corresponding to $\mu \in \calL$.
The \emph{Weil representation} is defined on the generators $S=(\smallSmatrix,\sqrt{\tau})$ and $T=(\smallTmatrix, 1)$
by the formulas
	\begin{equation}
\label{eq:weilrep}
      \begin{aligned}
	   \rho_\calL(T)\,\frake_\mu &= e(Q(\mu))\frake_\mu, \\
	   \rho_\calL(S)\,\frake_\mu &= \frac{e(\sgn(\calL))}{\sqrt{|\calL|}}
					   \sum_{\nu \in \calL}e(-(\mu,\nu))\frake_{\nu}.
       \end{aligned}
    \end{equation}
Let $N$ be the level of $L$.
Then it is well-known (see for instance Lemma 5.15 of \cite{StrombergWeilrep}) that elements of the form $(\gamma, \sqrt{c \tau + d})$, where $\gamma = \smallabcd \in \Gamma_0(N)$ act on $\frake_0$ by a character $\chi_\calL(\gamma)$. We do not need the general formula for this character. We just remark that if $M=\Z$ with quadratic form $x^2$, and $\calM=M'/M$, then the theta function $\Theta_M=\Theta$ transforms as $\Theta(\gamma\tau) = \sqrt{c\tau+d}\chi_{\calM}(\gamma)\Theta_M(\tau)$ and this transformation behavior is (usually) used to define a modular form of half-integral weight. 

For $\gamma \in \calL$, we define an associated theta function via
\begin{equation}
  \label{eq:theta-comp}
  \Theta_{L,\gamma}(\tau,\alpha,p) := v^{\frac{b^-}2 + m^-} \sum_{\lambda \in L + \gamma} p(\alpha(\lambda)) e(Q(\lambda_\alpha) \tau + Q(\lambda_{\alpha^\perp})\overline{\tau}).
\end{equation}
By Theorem 4.1 of \cite{boautgra}, we find the expansion
\begin{equation}
  \label{eq:theta}
  \Theta_L(\tau,\alpha,p) = \sum_{\gamma \in \calL} \Theta_{L,\gamma}(\tau,\alpha, p) e_\gamma
\end{equation}
transforms in $\tau$ as a vector-valued modular form of weight $(b^+-b^-)/2 + m^+-m^-$ for the Weil representation attached to $\calL$.
If $L$ is indefinite, then $\Theta_L$ is non-holomorphic in both variables.

Now let $L$ be the lattice of signature $(1,2)$ given by $\Z^3$ together with the quadratic form $Q(a,b,c) = -b^2+ac$ and let $D_0$ be a fundamental discriminant (note that the discriminants $D_0$ in Table \ref{tab:discs} are all fundamental). The discriminant group of $L$, $\calL = L'/L$, is isomorphic to $\Z/2\Z$ with quadratic form $-x^2/4$. Note that for $(a,b/2,c) \in L'$ with $a,b,c \in \Z$, we have that $[a,b,c]$ is an integral binary quadratic form of discriminant $-4Q(a,b,c)$. We also consider $\calL(D_0)$, the discriminant group of $D_0 L$ with $Q_{D_0}(a,b,c) = Q(a,b,c)/\abs{D_0}$ as quadratic form.  This is isomorphic to $\Z/4D_0\Z \oplus \Z/D_0\Z \oplus \Z/D_0\Z$ with the quadratic form $Q(a,b,c)/(4\abs{D_0})$.

Now let $N$ be a positive integer that is coprime to $D_0$ and consider the sublattice $M = M_N \subset L$ given by all vectors $(a,b,c)$ such that $N$ divides $a$. We remark that $\calM = M'/M$ is isomorphic to $L'/L \oplus \calN$, where $\calN = \Z/N\Z \oplus \Z/N\Z$. Moreover, if we equip $D_0 M$ with quadratic form $Q_{D_0}$, we obtain the discriminant group $\calL(D_0) \oplus \calN$ under the assumptions we made.  We will write $\frake_{\mu,\nu}$ the basis element of $\C[\calL(D_0) \oplus \calN] = \C[\calL(D_0)] \otimes \C[\calN]$\ corresponding to $\mu \in \calL(D_0)$ and $\nu \in \calN$.

The group $\SL_2(\Q)$ acts on the rational quadratic space $L \otimes_\Z \Q$ via isometries.
The action, which we denote by $\gamma.(a,b,c)$, is given by
\[
M(a,b,c) =
\left(\begin{smallmatrix}
  b & c \\
  -a & -b
\end{smallmatrix}
\right)
\mapsto
\gamma \left(\begin{smallmatrix}
  b & c \\
  -a & -b
\end{smallmatrix}\right)
\gamma^{-1},
\]
where we identify $(a,b,c)$ with the matrix $M(a,b,c)$.

As the genus character $\chi_{D_0}$ only depends on $a,b,c$ modulo $D_0$, we may view it as a function on $\calL(D_0)$. In \cite{AE}, it is shown that the linear map
\[
  \Psi_{D_0}: \C[\calL] \to \C[\calL(D_0)], \quad \frake_\mu \mapsto \sum_{\substack{\delta \in \calL(D_0) \\ \delta \equiv D_0\mu \smod{L} \\ Q_{D_0}(\delta) \equiv Q(\mu) \smod{\Z}}} \chi_{D_0}(\delta)\frake_\delta
\]
is an intertwiner for the Weil representations attached to $\calL^{\sgn{D_0}}$ and $\calL(D_0)$.
We consider the map
\[
 \Psi_{D_0,\calN}: \C[\calL] \to \C[\calL(D_0)\oplus\calN] = \C[\calL(D_0)] \otimes \C[\calN]
\]
obtained from $\Psi_{D_0}$ together with the natural inclusion $\C[\calL(D_0)] \hookrightarrow \C[\calL(D_0)] \otimes \C[\calN]$
given by $\frake_\delta \mapsto \frake_{\delta} \otimes \frake_0$.
The following lemma is crucial for us.
\begin{lemma}
\label{lem:twistvec}
  Let $\gamma :=
  \left(
  \begin{smallmatrix}
    a & b \\ c & d
  \end{smallmatrix}
  \right)
\in \Gamma_0(N) \cap \Gamma^0(4)$.
Let
\[
  v_{D_0} = \Psi_{D_0,\calN}(\frake_{0} + \frake_{1}) \in \C[\calL(D_0)] \otimes \C[\calN],
\]
where $\frake_0$ nd $\frake_1$ correspond to the elements $0,1 \in \Z/2\Z$.
 We have that
\[
  \rho_{\calM(D_0)}(\gamma,\sqrt{c\tau+d}) v_{D_0}= \nu_\theta(\gamma)^{\sgn(D_0)} v_{D_0}.
\]
\end{lemma}
\begin{proof}
Using the formulas for the Weil representation we see that $\rho_\calL(\gamma)$ for $\gamma \in \Gamma^0(4)$ acts on $\frake_0 + \frake_1$ by
$\overline{\chi_\calL(\gamma)}$. In fact, if $\gamma \in \Gamma^0(4)$, then $\widetilde{\gamma} = S^{-1}\gamma S \in \Gamma_0(4)$. Therefore,
\[
\rho_{\calL}(\gamma)(\frake_0 + \frake_1) = \rho_{\calL}(S\widetilde\gamma S^{-1}) (\frake_0 + \frake_1)  = \chi_\calL(\widetilde\gamma)(\frake_0 + \frake_1)
\]
by Lemma 5.15 in \cite{StrombergWeilrep} and \eqref{eq:weilrep}. Directly above the lemma, loc. cit., the direct relation to $\nu_\theta$ is also stated.
Since $\rho_\calN(\gamma)$ acts trivially on $\frake_0$ for $\gamma \in \Gamma_0(N)$, we are done.
\end{proof}

We obtain the following scalar-valued theta function.
\begin{proposition}
\label{prop:twisted-scalar-general}
  For $\calL$ and $\calM$ as above the scalar-valued theta function
\[
\Theta_{N,D_0}^*(\tau,\alpha,p) := \langle \Theta_{\calM(D_0)}(4\tau,\alpha,p), v_{D_0} \rangle
\]
satisfies
\[
\Theta_{N,D_0}(\gamma\tau,\alpha,p) = \sqrt{c\tau+d}^{2k}\, v_{\theta}^{-\sgn(D_0)}(\gamma)\, \Theta_{N,D_0}^*(\gamma\tau,\alpha,p)
\]
for all $\gamma = \smallabcd \in \Gamma_0(4N)$ with $k=-1/2 + m^+-m^-$ in $\tau$.
\end{proposition}
\begin{proof}
  The claim follows directly from Lemma \ref{lem:twistvec} by noting that if $\gamma = \smallabcd \in \Gamma_0(4N)$, then  
\[
 4\gamma\tau =
 \left(
 \begin{smallmatrix}
   a & 4b \\
   c/4 & d
 \end{smallmatrix}
 \right)
 (4\tau)
\]
and the matrix above is contained in $\Gamma_0(N) \cap \Gamma^0(4)$.
\end{proof}

Now we specialize these functions by choosing an isometry $\alpha_z$ of $L \otimes \R$ with
$\R^{1,2}$ or $\R^{2,1}$ for every $z \in \HH \cong \Gr(L)$.
Here, we use the identification $z \in \HH \cong \Gr(L)$ given by
\[
  z = x + iy \mapsto \R (-1, (z+\bar{z})/2, z\bar{z})
= (-1, x, x^2+y^2).
\]
Then we let $b_1 = b_1(z)$ be a normalized basis vector for the positive line $z$, i.e.,
\[
  b_1(z) := \frac{1}{y}(-1, x, x^2+y^2),
\]
and let
\[
  Z = Z(z) := \frac{1}{y}(-1,z,-z^2)
\]
and $b_2(z) = \Re (Z)$ and $b_3(z) = \Im(Z)$.
Note that we have 
\begin{equation}
  \label{eq:Ztrans}
  Z(\gamma z) = \left(\frac{cz+d}{c\bar{z}+d}\right)\gamma.Z(z) \text{ and } b_1(\gamma z) = \gamma.b_1(z),
\end{equation}
where $\gamma = \smallabcd$ acts on $z \in \HH$ via the usual fractional linear transformation.
We then let $\alpha_z(ab_1+bb_2+cb_3) = (a, b, c) \in \R^{1,2}$.
For integers $r,s \geq 0$ we let $P_{r,s}(a,b,c) = a^r(b+ic)^{s}$ (which is homogeneous of degree $(r,s)$) and define
\[
  p^*_{z,k}(X) := P_{1,k-1}(\alpha_z(X)) = (X,b_1) \cdot (X,b_2 + i b_3)^{k-1} = \frac{1}{y}(a\abs{z}^2 + b \Re(z) + c)(az^2+bz+c)^{k-1}
\]
for $X = (a,b,c) \in L \otimes \R$. 
We define 
\[
  \Theta_{1-k,N,D_0}^*(z, \tau) := y^{k-1}\Theta_{N,D_0}(\tau,\alpha_z,p_{z,k}^*).
\]
Similarly, we let
\[
  p_{z,k}(X) := \overline{P_{0,k}(X)} = (X,b_2 - i b_3)^{k} = (a\overline{z}^2+b\overline{z}+c)^{k}
\]
and define
$$
\Theta_{k,N,D_0}(z, \tau) := v^{-\frac{1}{2}}y^{-2k} \overline{\Theta_{N,D_0}(\tau,\alpha_z,p_{z,k})}.
$$

\begin{proposition}
  \label{prop:twisted-scalar}
  The theta function $\Theta_{1-k,N, D_0}^*(z, \tau)$ transforms like a modular form of weight $\frac{3}{2}-k$ for $\Gamma_0(4N)$ in $\tau$ 
  and as a modular form for $\Gamma_0(N)$ of weight $2-2k$. Similary, $\Theta_{k,N, D_0}(-\bar{z}, \tau)$ has weights $\frac{1}{2}+k$ and $2k$ in $\tau$ and $z$.
\end{proposition}
\begin{proof}
  The transformation behaviour in $\tau$ follows from Proposition \ref{prop:twisted-scalar-general}.
  The modularity in $z$ is a direct consequence of the fact that $\Gamma_0(N) \subset \SL_2(\Z)$ acts via isometries.
  It preserves the lattice $M$ and acts trivially on $\calM$ and since $\chi_{D_0}$ is invariant under $\SL_2(\Z)$,
  the claim follows from \eqref{eq:Ztrans}.
\end{proof}

Having shown modularity of the relevant theta functions, we now look at the scalar-valued theta functions $\Theta_{1-k}^*$ and $\Theta_k$ more closely, translating the above definitions into language using sums over binary quadratic forms which will be needed to connect with the main functions in this paper. 
We first define $Q_z$ for an integral binary quadratic form $Q=[a,b,c]$ as
$$
  Q_z:=\frac{1}{y}\left(a|z|^2+bx+c\right).
$$
The Fourier expansions (in $\tau$) of the theta functions are then easily seen to be given by
$$
\Theta_{1-k}^*(z,\tau)= \Theta_{1-k,N,D_0}^*\left(z,\tau\right)=v^k \sum_{D\in \Z} \sum_{\substack{Q=[a,b,c]\in \QQ_{D\left|D_0\right|}\\ N\mid a}} \chi_{D_0}(Q) Q_{z}Q(z,1)^{k-1}e^{-\frac{4\pi v}{\left|D_0\right|y^2}\left|Q\left(z,1\right)\right|^2}e^{-2\pi i D\tau},
$$
and
$$
\Theta_{k}\left(z,\tau\right)= \Theta_{k,N,D_0}\left(z,\tau\right)= y^{-2k}v^{\frac{1}{2}} \sum_{D\in \Z} \sum_{\substack{Q=[a,b,c]\in \QQ_{D\left|D_0\right|}\\ N\mid a}} \chi_{D_0}(Q) Q(z,1)^{k}e^{-\frac{4\pi Q_z^2 v}{\left|D_0\right|}}e^{2\pi i D\tau}.
$$
\begin{remark}
  \begin{enumerate}[leftmargin=*,label={\rm(\arabic*)}]
  \item The first theta function, $\Theta_{1-k}^*(z,\tau)$, is sometimes called the Millson theta function (cf. Section 2.6.2 of \cite{Alfes}) and other variants appeared e.g. in \cite{BruinierFunke, BKK, Hoevel}.
  \item The theta function $\Theta_{k}(z,\tau)$ is usually called the Shintani theta function in the literature. 
             A variant without the genus character was introduced by Shintani in \cite{Shintani}.
  \item Note that for trivial reasons (compatibility with the representation or rather degree of the polynomial and the behaviour of $\chi_{D_0}$ under $Q \mapsto -Q$), the theta functions $\Theta_{1-k}^*$ and $\Theta_k$ both vanish if $\sgn(D_0)(-1)^k=-1$, or in other words if $D_0 > 0$ and $k$ is odd or if $D_0 < 0$ and $k$ is even. 
  \end{enumerate}

\end{remark}
We next compute the Fourier expansions (in $\tau$) of $\Theta_k$ and $\Theta_{1-k}^*$ as $z$ approaches each cusp. 
\begin{lemma}\label{lem:Thcusps}
Let $\rho$ be a cusp of $\Gamma_0(N)$ and choose $M\in\SL_2(\Z)$ such that $M\infty=\rho$. Then 
\begin{equation}\label{eqn:Thcusp}
\Theta_{k}\left(z,\tau\right)\Big|_{2k,z}M = y^{-2k}v^{\frac{1}{2}} \sum_{D\in\Z}\sum_{\substack{Q=[a,b,c]\in \QQ_{D\left|D_0\right|}\\ Q\circ M^{-1}=\left[\alpha,\beta,\gamma\right]\\ N\mid \alpha}} \chi_{D_0}\left(Q\right) Q(z,1)^{k}e^{-\frac{4\pi Q_{z}^2v}{\left|D_0\right|}}e^{2\pi i D\tau}
\end{equation}
and 
\begin{equation}\label{eqn:Th*cusp}
\Theta_{1-k}^*\left(z,\tau\right)\Big|_{2-2k,z}M = v^{k} \sum_{D\in\Z}\sum_{\substack{Q=[a,b,c]\in \QQ_{D\left|D_0\right|}\\ Q\circ M^{-1}=\left[\alpha,\beta,\gamma\right]\\ N\mid \alpha}} \chi_{D_0}\left(Q\right) Q_z Q(z,1)^{k-1}e^{-\frac{4\pi v}{|D_0|y^2} |Q(z,1)|^2}e^{-2\pi i D\tau}.
\end{equation}

\end{lemma}
\begin{proof}
Since the arguments are similar, we only show \eqref{eqn:Thcusp}. A direct computation yields
$$
\Theta_{k}\left(z,\tau\right)\Big|_{2k,z}M = y^{-2k}v^{\frac{1}{2}} \sum_{D\in\Z}\sum_{\substack{Q=[a,b,c]\in \QQ_{D\left|D_0\right|}\\ N\mid a}} \chi_{D_0}(Q) Q\circ M (z,1)^{k}e^{-\frac{4\pi Q_{Mz}^2v}{\left|D_0\right|}}e^{2\pi i D\tau}.
$$
Moreover, $\chi_{D_0}(Q)=\chi_{D_0}\left(Q\circ M\right)$ for any $M\in \SL_2(\Z)$ and we have
$$
\frac{Q_{Mz}^2}{\left|D_0\right|}=\frac{\left|Q(Mz,1)\right|^2}{\left|D_0\right|\im\left(Mz\right)^2}-D=\frac{\left|Q(z,1)\right|^2}{\left|D_0\right|y^2}-D=\frac{\left(Q\circ M\right)_z^2}{\left|D_0\right|}.
$$
The claim hence follows.
\end{proof}

The two theta functions are related via the following differential equations. The following lemma follows mutatis mutandis as in the calculation in Lemma 3.3 of \cite{BKM}, after twisting by a genus character.
\begin{lemma}\label{lem:xi}
For every $k\geq 1$, we have
\begin{align}
\label{eqn:xiTh}
\xi_{k+\frac{1}{2},\tau}\left(\Theta_k\left(z,\tau\right)\right)&=-iy^{2-2k}\frac{\partial}{\partial z}\Theta_{1-k}^*\left(-\overline{z},\tau\right),\\
\label{eqn:xiTh*}
\xi_{\frac{3}{2}-k,\tau}\left(\Theta_{1-k}^*\left(z,\tau\right)\right)&=-iy^{2k}\frac{\partial}{\partial z}\Theta_k\left(z,\tau\right).
\end{align}
\end{lemma}

We now define two theta lifts using the theta kernels constructed above.
For a harmonic weak Maass form $H$ of weight $\frac{3}{2}-k$ for $\Gamma_0(4N)$, we consider the
regularized theta integral
$$
\Phi_{1-k}^*(H)(z)= \Phi_{1-k,N,D_0}^*(H)(z):=\left<H,\Theta_{1-k}^*\left(-\overline{z},\cdot\right)\right>^{\operatorname{reg}}.
$$
Similarly, for a harmonic weak Maass form $H$ of weight $\frac{1}{2}+k$ for $\Gamma_0(4N)$, we also consider
$$
\Phi_{k}(H)(z)= \Phi_{1-k,N,D_0}(H)(z):=\left<H,y^{-2k}\Theta_{k}\left(z,\cdot\right)\right>^{\operatorname{reg}}.
$$
The function $\Phi_{1-k}^*(H)$ is modular of weight $2-2k$ in $z$ and $\Phi_{k}(H)$ is modular of weight $2k$.
Here, the integrals are regularized as follows.

For two real analytic functions $F$ and $G$ satisfying weight $\kappa\in \frac{1}{2}\Z$ modularity for $\Gamma_0(4N)$, we define the \begin{it}regularized inner product\end{it}
$$
\left<F,G\right>^{\operatorname{reg}}:=\lim_{T\to\infty} \frac{1}{\left[\SL_2(\Z):\Gamma_0(4N)\right]}\int_{\Gamma_0(4N)\backslash \H_T} F(\tau)\overline{G(\tau)} v^{\kappa} \frac{du dv}{v^2},
$$
whenever it exists.
Here 
$$
\H_T:=\bigcup_{\gamma\in \SL_2(\Z)} \gamma \mathcal{F}_T,
$$
where we define the truncated fundamental domain for $\SL_2(\Z)$ by 
$$
\mathcal{F}_T:=\left\{ \tau\in \H: -\frac{1}{2}\leq u<\frac{1}{2}, v<T, |\tau|\geq 1, |\tau|=1\implies u\leq 0\right\}.
$$
Following the argument of Lemma 3.4 of \cite{BKM}, Lemma \ref{lem:xi} leads to the following lemma relating the regularized inner products of weak Maass forms against $\Theta_{k}$ and $\Theta_{1-k}^*$.
\begin{lemma}\label{lem:xireg}
Suppose that $D$ is a fundamental discriminant and $z\notin E_{D_0 D}$.  Then for every $s$ with $\re(s)>\max\left(1,\frac{k}{2}+\frac{3}{4}\right)$ one has 
\begin{align*}
\left<\xi_{k+\frac{1}{2}}\left(P_{k+\frac{1}{2},|D|,s}\right),\Theta_{1-k}^*\left(-\overline{z},\cdot\right)\right>^{\operatorname{reg}} = -\left<P_{k+\frac{1}{2},|D|,s},\xi_{\frac{3}{2}-k}\left(\Theta_{1-k}^*\left(-\overline{z},\cdot\right)\right)\right>^{\operatorname{reg}},\\ 
\left<\xi_{\frac{3}{2}-k}\left(P_{\frac{3}{2}-k,|D|,s}\right),\Theta_k\left(z,\cdot\right)\right>^{\operatorname{reg}} = -\left<P_{\frac{3}{2}-k,|D|,s},\xi_{k+\frac{1}{2}}\left(\Theta_k\left(z,\cdot\right)\right)\right>^{\operatorname{reg}}.
\end{align*}
\end{lemma}
\begin{proof}
By the argument in Lemma 3.4 of \cite{BKM}, one may reduce the first statement to showing that 
$$
\lim_{T\to\infty}\int_0^1P_{k+\frac{1}{2},|D|,s}\left(u+iT\right)\Theta^*\left(-\overline{z},u+iT\right)du =0
$$
as well as vanishing of similar integrals around other cusps of $\Gamma_0(4N)$.  For the cusps equivalent to $\infty$, $0$, and $\frac{1}{2}$, the argument follows by the usual bounds for the $M$-Whittaker function (as shown in \cite{BKM}).  The Poincar\'e series has no principal part at the cusps which are inequivalent to $0$, $\frac{1}{2}$, and $\infty$ and hence the corresponding integrals vanish.  The argument for the second statement is similar.
\end{proof}

The following diagram summarizes Lemmas \ref{lem:xi} and \ref{lem:xireg}, up to constants:
\[
\xymatrix{
&\Theta_{k}(z,\tau)\ar@{->}[dd]_{\xi_{k+\frac{1}{2},\tau}}\ar@/_/[ddr]_<<<<<<{\frac{\partial}{\partial z}}\ar@/^/[dl]_{\substack{\left<\xi_{\frac{3}{2}-k}\left(P_{\frac{3}{2}-k,|D|,s}\right) ,\cdot\right>\\ \vphantom{a_{a_n}}}}&\Theta_{1-k}^*(z,\tau)\ar@/^/[ddl]^<<<<<<{\frac{\partial}{\partial z}}\ar@{->}[dd]^{\xi_{\frac{3}{2}-k,\tau}} \ar@/_/[dr]^{\substack{\left<\xi_{k+\frac{1}{2}}\left(P_{k+\frac{1}{2},|D|,s}\right) ,\cdot\right>\\ \vphantom{a_{a_n}}  }}&&\\
\Phi_{k}\!\left(\xi_{\frac{3}{2}-k}\!\left(P_{\frac{3}{2}-k,|D|,s}\right)\right)\hspace{-.25in}&&&\hspace{-0.25in}\Phi_{1-k}^*\!\left(\xi_{k+\frac{1}{2}}\!\left(P_{k+\frac{1}{2}},|D|,s\right)\right)\\
&y^{2-2k}\frac{\partial}{\partial z}\Theta_{1-k}^*(-\overline{z},\tau)\ar@/_/[ul]^{\substack{\vphantom{a}\\ \left<P_{\frac{3}{2}-k,|D|,s} ,\cdot\right>\hphantom{aa} }} & y^{2k}\frac{\partial}{\partial z} \Theta_{k}(z,\tau)\ar@/^/[ur]_{\substack{\vphantom{a}\\ \left<P_{k+\frac{1}{2},|D|,s} ,\cdot\right>}}&
}
\]

\subsection{Theta lifts and spectral parameters}
\label{sec:theta-lifts}

Define (for $\tau=u+iv$, $z=x+iy$)
\begin{equation}\label{eqn:Fdefslarge}
\mathcal{F}_{1-k, N, D_0,D,s}(z) := \sum_{\substack{Q\in \QQ_{D_0D}\\ N\mid a}} \chi_{D_0}(Q) \sgn\left(Q_{z}\right) Q(z,1)^{k-1}\varphi_s^*\left(\frac{D_0Dy^2}{\left|Q(z,1)\right|^2}\right),
\end{equation}
where
$$
\varphi_s^*(w):=\frac{\Gamma\left(s+\frac{1}{4}\right)(4\pi)^{\frac{1}{4}}}{12\sqrt{\pi} \Gamma(2s)} w^{s-\frac{1}{4}}{_2F_1}\left(s-\frac{1}{4}, s-\frac{1}{4};2s;w\right).
$$
\begin{proposition}\label{prop:local}
For $\re(s)>1$, the function $\mathcal{F}_{0,N,D_0,D,s}$ is a local Maass form of weight $0$ and eigenvalue $4\lambda_{\frac{1}{2},s} = \lambda_{0,2s-\frac{1}{2}}$ on $\Gamma_0(N)$, up to the condition at the cusps.
\end{proposition}
\begin{remark}
Following an argument in \cite{BKK}, one could show that the functions $\mathcal{F}_{0,N,D_0,D,s}$ satisfy the necessary condition at the cusps, but we only need this for one special eigenvalue, so we do not work this out here.
\end{remark}
\begin{proof}[Proof of Proposition \ref{prop:local}]
We compute $\Phi_{1-k}^*\left(P_{\frac{3}{2}-k,|D|,s}\right)$ as in \cite{BKM}, where $P_{\frac{3}{2}-k,|D|,s}$ was defined in \eqref{eqn:Poincdef}.  Unfolding the Poincar\'e series and following the proof of Theorem 1.3 (2) in \cite{BKM}, we obtain 
\begin{multline*}
\Phi_{1-k}^*\left(P_{\frac{3}{2}-k,|D|,s}\right)(z) =\frac{1}{\left[\SL_2(\Z):\Gamma_0(4N)\right]}\left(4\pi |D|\right)^{\frac{1}{4}-\frac{k}{2}}\Gamma(2s)^{-1}\\
\times \sum_{\substack{Q\in \QQ_{DD_0}\\ N\mid a}} \chi_{D_0}(Q)Q_zQ(z,1)^{k-1}\mathcal{I}\left(\frac{D_0Dy^2}{\left|Q\left(z,1\right)\right|^2}\right),
\end{multline*}
where 
$$
\mathcal{I}(w):=\int_{0}^{\infty} \mathcal{M}_{\frac{3}{2}-k,s}\left(-v\right)e^{\frac{v}{2}}v^{-\frac{1}{2}}e^{-\frac{v}{w}}dv.
$$
In the proof of Theorem 1.3 (2) of \cite{BKM}, this integral was computed to be
$$
\mathcal{I}(w)=\Gamma\left(s+\frac{k}{2}-\frac14\right)\left(1-w\right)^{-\frac{1}{2}}\left(w\right)^{s+\frac{k}{2}-\frac{1}{4}}{_2F_1}\left(s-\frac{k}{2}+\frac{1}{4},s+\frac{k}{2}-\frac{3}{4};2s;w\right).
$$
We now plug in $w=\frac{D_0Dy^2}{\left|Q(z,1)\right|^2}$, use 
\begin{equation}\label{eqn:QQz}
\left|Q(z,1)\right|^2=\left|D_0\right|D y^2+Q_z^2y^2.
\end{equation}
 to rewrite 
$$
\left(1-\frac{D_0D y^2}{\left|Q(z,1)\right|^2}\right)^{-\frac{1}{2}}=\frac{\sqrt{D_0 D}}{\left|Q_{z}\right|},
$$
and plug in $k=1$ to yield
\begin{equation}\label{eqn:Phi1/2}
\Phi_{0}^*\left(P_{\frac{1}{2},|D|,s}\right)(z) = \frac{\left(D_0D\right)^{\frac{1}{2}}}{\left[\SL_2(\Z):\Gamma_0(4N)\right]\left(4\pi\left|D\right|\right)^{\frac{1}{4}}\Gamma(2s)} \sum_{\substack{Q\in \QQ_{DD_0}\\ N\mid a}} \chi_{D_0}(Q)\sgn\left(Q_z\right)\varphi_{s}^*\left(\frac{D_0Dy^2}{\left|Q(z,1)\right|^2}\right).
\end{equation}
The claim is hence equivalent to showing that for $H=P_{\frac{3}{2}-k,|D|,s}$, the function $\Phi_0^*(H)$ is a local Maass form with eigenvalue $4\lambda_{\frac{3}{2}-k,s}=\lambda_{2-2k,2s-\frac{1}{2}}$.

However, by Lemma \ref{lem:xi} and Lemma \ref{lem:xireg} together with the fact that for $\kappa\in\frac{1}{2}\Z$
$$
\Delta_{\kappa} = -\xi_{2-\kappa}\xi_{\kappa},
$$
we have that for $z\notin E_D$
\begin{multline}\label{eqn:DeltaH}
\Delta_{2-2k}\left(\Phi_{1-k}^*(H)\right) = -\xi_{2k}\xi_{2-2k}\left(\left<H,\Theta_{1-k}^*\left(-\overline{z},\cdot\right)\right>^{\operatorname{reg}}\right)=2\xi_{2k}\left(\left<H,iy^{2-2k}\frac{\partial}{\partial z}\left(\Theta_{1-k}^*\left(-\overline{z},\cdot\right)\right)\right>^{\operatorname{reg}}\right)\\
=-2\xi_{2k}\left(\left<H,\xi_{k+\frac{1}{2}}\left(\Theta_{k}\left(z,\cdot\right)\right)\right>^{\operatorname{reg}}\right)=-4\left<\xi_{\frac{3}{2}-k}(H),iy^{2k}\frac{\partial}{\partial z}\left(\Theta_{k}\left(z,\cdot\right)\right)\right>^{\operatorname{reg}}\\
=4\left<\Delta_{k+\frac{1}{2}}(H),\Theta_{1-k}^*\left(-\overline{z},\cdot\right)\right>^{\operatorname{reg}}=4\lambda_{\frac{3}{2}-k,s}\Phi_{1-k}^*(H).
\end{multline}

\end{proof}

We are mainly interested in the analytic continuation $\mathcal{F}_{0,N,D_0,D}:=\left[\mathcal{F}_{0,N,D_0,D,s}\right]_{s=\frac{3}{4}}$ to $s=\frac{3}{4}$.  
\begin{corollary}\label{cor:FchiD}
The function $\mathcal{F}_{0,N,D_0,D}$ is a locally harmonic Maass form of weight $0$ on $\Gamma_0(N)$.
\end{corollary}
\begin{proof}
Recall that the Poincar\'e series $P_{\frac{1}{2},|D|,s}$ have an analytic continuation to $s=\frac{3}{4}$, which we have denoted $P_{\frac{1}{2},|D|}$ and which is harmonic.  Since the theta lift $\Phi_0^*\left(P_{\frac{1}{2},|D|}\right)$ exists, we may analytically continue $\mathcal{F}_{0,N,D_0,D,s}$ to $s=\frac{3}{4}$ by \eqref{eqn:Phi1/2}.  Moreover, since the eigenvalue of $\Phi_0^*\left(P_{\frac{1}{2},|D|,s}\right)$ is $\lambda_{0,2s-\frac{1}{2}}=4\lambda_{\frac{1}{2},s}$ and $\lambda_{\frac{1}{2},\frac{3}{4}}=0$, we obtain that $\mathcal{F}_{0,N,D_0,D}$ is a locally harmonic Maass form of weight $0$.

It remains to show the growth condition at all cusps. In order to show that $\mathcal{F}_{0,N,D_0,D}$ is bounded at each cusp, we use \eqref{eqn:Th*cusp} and pull the operator inside the theta lift and then using the same argument as that used to obtain \eqref{eqn:Phi1/2} yields for $H=P_{\frac{3}{2}-k,|D|,s}$ that 
\begin{multline*}
\mathcal{F}_{0,N,D_0,D,M}(z):=\mathcal{F}_{0,N,D_0,D}\Big|_{0}M(z)= \Phi_{0}^*(H)\Big|_{2-2k}M(z) \\
= \sum_{\substack{Q=[a,b,c]\in \QQ_{DD_0}\\ Q\circ M^{-1}=\left[\alpha,\beta,\gamma\right]\\ N\mid \alpha}} \chi_{D_0}(Q)\sgn(Q_z) \varphi_{\frac{3}{4}}^*\left(\frac{D_0Dy^2}{\left|Q(z,1)\right|^2}\right).
\end{multline*}
Similarly, for the holomorphic modular form $h=P_{k+\frac{1}{2},|D|}$, we define
\[
\Phi_{k}(h):=\left<h,\Theta_{k}\left(z,\cdot\right)\right>.
\]
and for a discriminant $D=-m$, we then define
\begin{equation}\label{eqn:fchiDdef}
\f_{1,N,D_0,D}:=\left[\SL_2(\Z):\Gamma_0(4N)\right]\Phi_1\left(P_{\frac{3}{2},|D|}\right).
\end{equation}
Note that the exponential decay of the cusp form $P_{\frac{3}{2},|D|}$ and the polynomial growth of $\Theta_1$ in the integration variable implies that regularization is unnecessary for the cusp form $H=P_{\frac{3}{2},|D|}$.  We next compute
\begin{multline*}
\f_{1,N,D_0,D,M}(z):=\f_{1,N,D_0,D}\Big|_{2} M(z)= \Phi_{1}(h)\Big|_{2k} M(z)\\
 = \sum_{\substack{Q=[a,b,c]\in \QQ_{DD_0}\\ Q\circ M^{-1}=\left[\alpha,\beta,\gamma\right]\\ N\mid \alpha}} \chi_{D_0}(Q)\sgn(Q_z) \varphi_{\frac{3}{4}}^*\left(\frac{D_0Dy^2}{\left|Q(z,1)\right|^2}\right).
\end{multline*}

By \eqref{eqn:xiTh} and Lemma \ref{lem:xireg}, for a weight $\frac{1}{2}$ harmonic weak Maass form, we have 
\[
i\frac{\partial}{\partial z}\Phi_{0}^*\left(H\right)= \left<H,-i\frac{\partial}{\partial z}\Theta_{0}\left(-\overline{z},\cdot\right)\right>=\left<H,\xi_{\frac{3}{2}}\left(\Theta_{1}\left(z,\cdot\right)\right)\right>=\left<\xi_{\frac{1}{2}}\left(H\right), \Theta_{1}\left(z,\cdot\right)\right>= \Phi_{1}\left( \xi_{\frac{1}{2}}\left(H\right)\right).
\]
Hence for every weight $\frac{1}{2}$ harmonic weak Maass form $H$, we have
\begin{equation}\label{eqn:DPhi*}
\mathcal{D}\left(\Phi_0^*(H)\right) = -\frac{1}{2\pi}\Phi_1\left(\xi_{\frac{1}{2}}\left(H\right)\right).
\end{equation}
We have also seen that (see the second and third identities of \eqref{eqn:DeltaH} together with Lemma \ref{lem:xireg}) 
\begin{equation}\label{eqn:xiPhi*}
\xi_{0}\left(\Phi_{0}^*(H)\right) = 2\Phi_{1}\left(\xi_{\frac{1}{2}}(H)\right).
\end{equation}
Since $\xi_{\frac{1}{2}}(H)= ch$ for some constant $c\in\R$ and slashing commutes with the differential operators, we see that 
\[
\mathcal{F}_{0,N,D_0,D,M}-2c \left( f_{1,D_0,D,M}^* -\frac{1}{4\pi} \mathcal{E}_{f_{1,D_0,D,M}}\right) 
\]
is annihilated by both $\mathcal{D}$ and $\xi_{0}$, where we recall the Eichler integrals $f^*$ and $\mathcal{E}_f$ for cusp forms $f$ defined in \eqref{eqn:Eichnonhol} and \eqref{eqn:Eichhol}, respectively and extend their definition via Fourier expansions for Eisenstein series. Thus we see that the above difference is both holomorphic and anti-holomorphic, and hence a local constant. The functions $ f_{1,D_0,D,M}^*$ and $\mathcal{E}_{f_{1,D_0,D,M}}$ both grow at most polynomially towards the cusps because the Eichler integrals corresponding to the cuspidal component vanish at the cusps while the harmonic Eisenstein series component may have polynomial growth. Hence $\mathcal{F}_{0,N,D_0,D}$ grows at most polynomially towards the cusps, yielding the claim.

\end{proof}

\subsection{Connection between the two analytic continuations}

It turns out that the function $\f_{1,N,D_0,D}$ defined in \eqref{eqn:fchiDdef} essentially equals Kohnen's $f_{1,N,D,D_0}$.
\begin{lemma}\label{lem:f1fchi}
We have 
$$
f_{1,N,D,D_0}=\frac{6\left(4\pi\right)^{\frac{1}{4}}\Gamma\left(\frac{3}{2}\right)}{|D|^{\frac{3}{4}}}\f_{1,N,D_0,D}.
$$
\end{lemma}

\begin{proof}
For $\re(s)>\frac{3}{4}$, we first define 
$$
\f_{1,N,D_0,D,s}(z) := \left[\Gamma_0(4):\Gamma_0(4N)\right] \Phi_{1}\left(P_{\frac{3}{2},|D|,s}\right),
$$
where 
$$
\Phi_{k}(H):=\left<H,\Theta_{k}\left(z,\cdot\right)\right>.
$$
Recalling that $P_{\frac{3}{2},|D|} = \left[P_{\frac{3}{2},|D|,s}\right]_{s=\frac{3}{4}}$, we have 
$$
\f_{1,N,D_0,D}= \left[\f_{1,N,D_0,D,s}\right]_{s=\frac{3}{4}}.
$$
Following the proof of Proposition \ref{prop:local}, for $\re(s)>\frac{3}{4}$ we obtain
$$
\f_{1,N,D_0,D,s}(z) =\sum_{\substack{Q\in \QQ_{D_0D}\\ N\mid a}} \frac{\chi_{D_0}(Q)}{Q(z,1)} \varphi_s\left(\frac{D_0Dy^2}{\left|Q(z,1)\right|^2}\right),
$$
where 
$$
\varphi_s(w):=\frac{\Gamma\left(s+\frac{1}{4}\right)|D|^{\frac{3}{4}}}{6\left(4\pi\right)^{\frac{1}{4}}\Gamma(2s)} w^{s-\frac{3}{4}}{_2F_1}\left(s+\frac{1}{4}, s-\frac{3}{4};2s;w\right).
$$
For $\re(s)>\frac{3}{4}$, we furthermore define 
$$
f_{1,N,D,D_0,s}(z):=\sum_{\substack{Q\in \QQ_{D_0D}\\ N\mid a}} \frac{\chi_{D_0}(Q)}{Q(z,1)\left|Q(z,1)\right|^{2s-\frac{3}{2}}}.
$$
Kohnen \cite{KohnenCoeff} used Hecke's trick to show that the analytic continuation of $f_{1,N,D,D_0,s}$ to $s=\frac{3}{4}$ exists and equals $f_{1,N,D,D_0}$.  We now set
$$
\alpha_t(w):=w^{\frac{3}{4}-t} \frac{\Gamma(2t)}{\Gamma\left(t+\frac{1}{4}\right)} \varphi_t(w).
$$
Since ${_2F_1}\left(1,0;\frac{3}{2};w\right)=1$, we see that 
$$
\alpha_{\frac{3}{4}}(w) = \frac{|D|^{\frac{3}{4}}}{6\left(4\pi\right)^{\frac{1}{4}}}
$$
is independent of $w$.  Hence we have
\begin{multline}\label{eqn:f1fchi1}
\frac{|D|^{\frac{3}{4}}}{6\left(4\pi\right)^{\frac{1}{4}}\Gamma\left(\frac{3}{2}\right)}f_{1,N,D,D_0}(z)\\
=\Bigg[\frac{\Gamma\left(s+\frac{1}{4}\right)}{\left(D_0Dy^2\right)^{\frac{3}{2}-2s}\Gamma(2s)}\sum_{\substack{Q\in \QQ_{D_0D}\\ N\mid a}} \frac{\chi_{D_0}(Q)}{Q(z,1)\left|Q(z,1)\right|^{2s-\frac{3}{2}}}\alpha_{\frac{3}{4}}\left(\frac{D_0Dy^2}{\left|Q(z,1)\right|^2}\right)\Bigg]_{s=\frac{3}{4}}.
\end{multline}
Therefore, we conclude that
\begin{multline}\label{eqn:diff}
\f_{1,N,D_0,D}(z) - \frac{|D|^{\frac{3}{4}}}{6\left(4\pi\right)^{\frac{1}{4}}\Gamma\left(\frac{3}{2}\right)}f_{1,N,D,D_0}(z)\\
= \Bigg[\frac{\Gamma\left(s+\frac{1}{4}\right)}{\left(D_0Dy^2\right)^{\frac{3}{2}-2s}\Gamma(2s)}\sum_{\substack{Q\in \QQ_{D_0D}\\ N\mid a}} \frac{\chi_{D_0}(Q)}{Q(z,1)\left|Q(z,1)\right|^{2s-\frac{3}{2}}}\left(\alpha_s\left(\frac{D_0Dy^2}{\left|Q(z,1)\right|^2}\right)-\alpha_{\frac{3}{4}}\left(\frac{D_0Dy^2}{\left|Q(z,1)\right|^2}\right)\right)\Bigg]_{s=\frac{3}{4}}\!\!\!.
\end{multline}
We next show that the right-hand side of \eqref{eqn:diff} converges absolutely for $\re(s)>-\frac{5}{4}$.  Hence its analytic continuation to $s=\frac{3}{4}$ is simply its value with $s=\frac{3}{4}$ plugged in, which is clearly zero.

Recall that by the power series expansion for ${_2F_1}\left(s+\frac{1}{4}, s-\frac{3}{4};2s;w\right)$, for $|w|<1$ we have 
$$
\alpha_s(w) - \alpha_{\frac{3}{4}}(w)= \frac{|D|^{\frac{3}{4}}}{6\left(4\pi\right)^{\frac{1}{4}}}\left(\frac{\left(s+\frac{1}{4}\right)\left(s-\frac{3}{4}\right)}{2s}\right) w + O\left(w^2\right).
$$
Using \eqref{eqn:QQz}, we have $|w|\leq 1$, with $|w|=1$ if and only if $Q_z=0$.  However, by Lemma 5.1 (1) of \cite{BKK}, for each $z\in \H$, there are only finitely many $Q\in \mathcal{Q}_{D_0D}$ for which $Q_z=0$.  For these finitely many $Q$, we may directly plug in $s=\frac{3}{4}$ to see that these terms vanish on the right-hand side of \eqref{eqn:diff}.  We hence conclude that the sum on the right-hand side of \eqref{eqn:diff} converges absolutely for $\re(s)>-\frac{5}{4}$, which concludes the proof.
\end{proof}

\section{Proofs of Theorem \ref{mainthm} and Corollaries \ref{maincor}, \ref{cor:SDodd}, and \ref{cor:sumtwocubes}}\label{sec:mainproofs}
Here we tie together the results of the previous sections to complete the proofs of Theorem \ref{mainthm} and Corollaries \ref{maincor}, \ref{cor:SDodd}, and \ref{cor:sumtwocubes}.

\subsection{Computation of Local Polynomials}
The basic idea of Theorem \ref{mainthm} is to relate the (logarithmic) singularities of $\mathcal{F}_{0,N,D_0,D}$ along geodesics with the quantities $F_{0,N,D,D_0}(x)$. This ties together the invariance of $\mathcal{F}_{0,N,D_0,D}$ with that of $F_{0,N,D,D_0}$, but only ``part of'' $\mathcal{F}_{0,N,D_0,D}$ (a locally constant function) contributes to the sum $F_{0,N,D,D_0}$ defined in \eqref{eqn:F1}. It is hence of interest to compute the locally constant function which gives these singularities. In general, for $k\geq 1$, the function exhibiting the singularity is a local polynomial. We thus define 
\begin{multline}\label{eqn:Pdef}
\mathcal{P}_{D_0,D,s}(z):=-\sum_{\substack{ t> 0\\  \exists Q\in \QQ_{D_0D}\\ Q_{z+it}=0>a\\ N\mid a}}\lim_{\varepsilon\to 0^+}\left(\mathcal{F}_{0,N,D_0,D,s}\left(z+it+i\varepsilon\right)-\mathcal{F}_{0,N,D_0,D,s}\left(z+it-i\varepsilon\right)\right)\\
-\lim_{\varepsilon\to 0^+}\left(\mathcal{F}_{0,N, D_0,D,s}\left(z+i\varepsilon\right)-\mathcal{F}_{0,N, D_0,D,s}\left(z\right)\right)
\end{multline}
and 
$$
\mathcal{G}_{D_0,D,s}:=\mathcal{F}_{0,N, D_0,D,s}-\mathcal{P}_{D_0,D,s}.
$$
Here we suppress the dependence on $N$ and the weight in the definitions for ease of notation in the following calculations. The proof below follows the argument in Theorem 7.1 of \cite{BKK}, but a more direct approach via theta lifts and written in terms of wall-crossing behvior across Weyl chambers may be found in Satz 3.15 of \cite{Hoevel}. 
\begin{lemma}\label{lem:Gcont}
If $DD_0$ is not a square, then the function $\mathcal{G}_{0,D_0,D,s}$ is continuous.  
\end{lemma}
\begin{proof}
In each connected component of $\HH\setminus E_D$, the function $\mathcal{F}_{0,N,D_0,D,s}$ is continuous and $\mathcal{P}_{D_0,D,s}$ is constant.  Suppose now that $z\in E_{D}$.  By continuity in each connected component, it suffices to show that for each sign $\pm$
$$
\lim_{\delta\to 0^+} \mathcal{G}_{D_0,D,s}\left(z\pm i\delta\right) = \mathcal{G}_{D_0,D,s}(z).
$$
Using \eqref{eqn:Pdef} and continuity at $z+i\delta$ for $\delta>0$, we compute (in the following we suppress the dependence on $N$ and the weight $0$ in $\mathcal{F}_{0,N,D_0,D,s}$ for ease of notation)
\begin{multline*}
\lim_{\delta \to 0^+}\!\left(\mathcal{P}_{D_0,D,s}\!\left(z+ i\delta\right) - \mathcal{P}_{D_0,D,s}(z)\right) =\! -\hspace{-.1in}\sum_{\substack{ t> 0\\  \exists Q\in \QQ_{D_0D}\\ Q_{z+it}=0>a\\ N\mid a}}\!\!\lim_{\varepsilon\to 0^+}\!\left(\mathcal{F}_{D_0,D,s}\!\left(z+it+i\varepsilon\right)-\mathcal{F}_{D_0,D,s}\!\left(z+it-i\varepsilon\right)\right)\\
+\!\!\sum_{\substack{ t> 0\\  \exists Q\in \QQ_{D_0D}\\ Q_{z+it}=0>a\\ N\mid a}}\!\!\lim_{\varepsilon\to 0^+}\!\left(\mathcal{F}_{D_0,D,s}\!\left(z+it+i\varepsilon\right)-\mathcal{F}_{D_0,D,s}\!\left(z+it-i\varepsilon\right)\right)+\lim_{\varepsilon\to 0^+}\!\left(\mathcal{F}_{D_0,D,s}\!\left(z+i\varepsilon\right)-\mathcal{F}_{D_0,D,s}\!\left(z\right)\right)\\
= \lim_{\varepsilon\to 0^+}\!\left(\mathcal{F}_{D_0,D,s}\!\left(z+i\varepsilon\right)-\mathcal{F}_{D_0,D,s}\!\left(z\right)\right).
\end{multline*}
Similarly, we have 
\begin{multline*}
\lim_{\delta \to 0^+}\!\left(\mathcal{P}_{D_0,D,s}\!\left(z- i\delta\right) - \mathcal{P}_{D_0,D,s}(z)\right)=\! -\hspace{-.1in}\sum_{\substack{ t\geq 0\\  \exists Q\in \QQ_{D_0D}\\ Q_{z+it}= 0>a\\ N\mid a}}\!\!\lim_{\varepsilon\to 0^+}\!\left(\mathcal{F}_{D_0,D,s}\!\left(z+it+i\varepsilon\right)-\mathcal{F}_{D_0,D,s}\!\left(z+it-i\varepsilon\right)\right)\\
+\!\!\sum_{\substack{ t> 0\\  \exists Q\in \QQ_{D_0D}\\ Q_{z+it}=0>a\\ N\mid a}}\!\!\lim_{\varepsilon\to 0^+}\!\left(\mathcal{F}_{D_0,D,s}\!\left(z+it+i\varepsilon\right)-\mathcal{F}_{D_0,D,s}\!\left(z+it-i\varepsilon\right)\right)+\lim_{\varepsilon\to 0^+}\!\left(\mathcal{F}_{D_0,D,s}\!\left(z+i\varepsilon\right)-\mathcal{F}_{D_0,D,s}\!\left(z\right)\right)\\
= \lim_{\varepsilon\to 0^+}\left(\mathcal{F}_{D_0,D,s}\left(z-i\varepsilon\right)-\mathcal{F}_{D_0,D,s}\left(z\right)\right).
\end{multline*}

\end{proof}

We now explicitly compute $\mathcal{P}_{D_0,D,s}$.
\begin{proposition}\label{prop:Pval}
If $\re(s)\geq \frac{3}{4}$, then 
\begin{equation}\label{eqn:Pval}
\mathcal{P}_{D_0,D,s}(z)= 4 \varphi_s^*(1)\sum_{\substack{Q\in \QQ_{D_0D}\\ Q_{z}>0>a\\ N\mid a}} \chi_{D_0}(Q)+2\varphi_s^*(1)\sum_{\substack{Q\in \QQ_{D_0D}\\ Q_{z}=0>a\\ N\mid a}} \chi_{D_0}(Q).
\end{equation}
\end{proposition}
\begin{proof}
Note that both sides of \eqref{eqn:Pval} are analytic in for $\re(s)>\frac{3}{4}$ and the limit $s\to \frac{3}{4}^+$ exists.  Hence by the identity theorem, it is sufficient to prove the identity for $\re(s)$ sufficiently large.  

We use the representation \eqref{eqn:Fdefslarge} to explicitly compute $\mathcal{P}_{D_0,D,s}$ for $s$ sufficiently large.  To compute $\mathcal{P}_{D_0,D,s}$, for $z\in E_D$ we compute
\begin{multline*}
\lim_{\varepsilon\to 0^+} \left(\mathcal{F}_{0,N,D_0,D,s}\left(z+i\varepsilon\right)-\mathcal{F}_{0,N,D_0,D,s}\left(z-i\varepsilon\right)\right) = 2\varphi_s^*\left(1\right) \sum_{\substack{Q\in \QQ_{D_0D}\\ Q_{z}=0\\ a<0\\ N\mid a}} \left(\chi_{D_0}\left(-Q\right)-\chi_{D_0}(Q)\right)\\
=-4\varphi_s^*\left(1\right) \sum_{\substack{Q\in \QQ_{D_0D}\\ Q_{z}=0\\ a<0\\ N\mid a}} \chi_{D_0}(Q),
\end{multline*}
where the last line follows from $D_0<0$ so $\chi_{D_0}(-Q)=-\chi_{D_0}(Q)$.   Furthermore, with the convention $\sgn(0)=0$ we obtain
\begin{multline*}
\lim_{\varepsilon\to 0^+} \left(\mathcal{F}_{0,N,D_0,D,s}\left(z+i\varepsilon\right)-\mathcal{F}_{0,N,D_0,D,s}\left(z\right)\right) = \varphi_s^*\left(1\right) \sum_{\substack{Q\in \QQ_{D_0D}\\ Q_{z}=0\\ a<0\\ N\mid a}} \left(\chi_{D_0}\left(-Q\right)-\chi_{D_0}(Q)\right)\\
=-2\varphi_s^*\left(1\right) \sum_{\substack{Q\in \QQ_{D_0D}\\ Q_{z}=0\\ a<0\\ N\mid a}} \chi_{D_0}(Q).
\end{multline*}
\end{proof}

\subsection{Proof of Theorem \ref{mainthm}}

We now have all of the preliminaries necessary to prove Theorem \ref{mainthm}.
\begin{proof}[Proof of Theorem \ref{mainthm}]

Suppose that $D_0D$ is not a square.  Then by Lemma \ref{lem:Gcont}, $\mathcal{G}_{D_0,D,\frac{3}{4}}$ is continuous.  Proposition \ref{prop:Pval} then implies that the locally constant part of $\mathcal{F}_{0,N,D_0,D}$ is (since $\varphi_{s}$ is continuous as a function of $s$) 
\[
4 \varphi_{\frac{3}{4}}^*(1)\sum_{\substack{Q\in \QQ_{D_0D}\\ Q_{z}>0>a\\ N\mid a}} \chi_{D_0}(Q)+2 \varphi_{\frac{3}{4}}^*(1)\sum_{\substack{Q\in \QQ_{D_0D}\\ Q_{z}=0>a\\ N \mid a}} \chi_{D_0}(Q).
\]
Taking the limit $z\to x\in \Q$, the second sum becomes empty and $Q_z>0$ becomes $Q(x,1)>0$.  This yields the function 
\begin{equation}\label{eqn:Px}
4 \varphi_{\frac{3}{4}}^*(1)\sum_{\substack{Q\in \QQ_{D_0D}\\ ax^2+bx+c>0>a\\ N\mid a}} \chi_{D_0}(Q).
\end{equation}
We now claim that \eqref{eqn:Px} is $\Gamma_0(N)$-invariant if and only if $\f_{1,N,D_0,D}$ is in the space spanned by Eisenstein series.    

Suppose that $\f_{1,N,D_0,D}$ is in the space spanned by Eisenstein series and choose a linear combination of Maass-Eisenstein series $\mathcal{E}_{0,N,D_0,D}$ for which 
\[
\xi_0\left(\mathcal{E}_{0,N,D_0,D}\right) = \f_{1,N,D_0,D}.
\]
Since $\xi_{0}\left(\mathcal{F}_{0,N,D_0,D}\right) = \f_{1,N,D_0,{D}}$ by Lemma \ref{lem:xireg}, we obtain that $\mathcal{F}_{0,N,D_0,D}-\mathcal{E}_{0,N,D_0,D}$ is locally holomorphic and $\Gamma_0(N)$-invariant.

We now show that the operator $\mathcal{D}:=\frac{1}{2\pi i}\frac{\partial}{\partial z}$ also annihilates $\mathcal{F}_{0,N,D_0,D}-\mathcal{E}_{0,N,D_0,D}$, from which we obtain that $\mathcal{F}_{0,N,D_0,D}-\mathcal{E}_{0,N,D_0,D}$ is a $\Gamma_0(N)$-invariant local constant.  For this, we recall \eqref{eqn:DPhi*} and \eqref{eqn:xiPhi*}. Since there exist non-zero constants $c_1$ and $c_2$ for which $\mathcal{F}_{0,N,D_0,D}=c_1\Phi_0^*\left(P_{\frac{1}{2},|D|}\right)$ by \eqref{eqn:Phi1/2} and $\f_{1,N,D_0,D}=c_2\Phi_1\left(\xi_{\frac{1}{2}}\left(P_{\frac{1}{2},|D|}\right)\right)$ by \eqref{eqn:fchiDdef}, we conclude that
\[
\xi_{0}\left(\mathcal{F}_{0,N,D_0,D}\right) =\frac{2c_1}{c_2}\f_{1,N,D_0,D}
\]
and 
\[
\mathcal{D}\left(\mathcal{F}_{0,N,D_0,D}\right)=- \left(\frac{c_1}{2\pi c_2}\right)\f_{1,N,D_0,D}.
\]
Therefore \eqref{eqn:FxiD} (and hence \eqref{decomplocalmaass}) holds and the ratio of the constants in \eqref{decomplocalmaass} is 
\[
\frac{\alpha_D}{\beta_D}=-\frac{1}{4\pi}.
\]
  We next recall that the ratio of the constants for the Eisenstein series are identical.   All of the $\Gamma_0(N)$-invariant Maass-Eisenstein series are of the form 
\[
\left[\sum_{M\in \Gamma_{\infty}\backslash\Gamma_0(N)} y^s\Big|_0 M\gamma\right]_{s=1}
\]
for some $\gamma\in \Gamma_0(N)\backslash \SL_2(\Z)$.  Since $\xi_0$ and $\mathcal{D}$ commute with the slash operator, we have 
\begin{align*}
\xi_0\left( \sum_{M\in \Gamma_{\infty}\backslash\Gamma_0(N)} y^s\Big|_0 M\gamma \right)& = s\sum_{M\in \Gamma_{\infty}\backslash\Gamma_0(N)} y^{s-1}\Big|_0 M\gamma,\\
\mathcal{D}\left( \sum_{M\in \Gamma_{\infty}\backslash\Gamma_0(N)} y^s\Big|_0 M\gamma \right)& =-\frac{s}{4\pi}\sum_{M\in \Gamma_{\infty}\backslash\Gamma_0(N)} y^{s-1}\Big|_0 M\gamma.
\end{align*}
It follows that 
\[
\mathcal{D}\left(\mathcal{E}_{0,N,D_0,D}\right) = -\frac{1}{4\pi} \xi_{0}\left(\mathcal{E}_{0,N,D_0,D}\right)=-\frac{1}{4\pi} \f_{1,N,D_0,D}.
\]
Hence we see that the difference $\mathcal{F}_{0,N,D_0,D}-\mathcal{E}_{0,N,D_0,D}$ is locally holomorphic and annihilated by $\mathcal{D}$ away from $E_D$.  It follows that $\mathcal{F}_{0,N,D_0,D}-\mathcal{E}_{0,N,D_0,D}$ is a $\Gamma_0(N)$-invariant local constant.  

Since $\mathcal{E}_{0,N,D_0,D}$ does not contribute to the polynomial part, we conclude that the polynomial part of $\mathcal{F}_{0,N,D_0,D}$ is $\Gamma_0(N)$-invariant.  

Conversely, assume that $\f_{1,N,D_0,D}$ is not in the space spanned by Eisenstein series.  By the above argument, we may subtract a linear combination of Eisenstein series without affecting the polynomial part, so that we may assume without loss of generality that $\f_{1,N,D_0,D}$ is cuspidal.  Since $S_{2}(N)$ is one-dimensional, we have that $\f_{1,N,D_0,D}=a_D f$, where $a_D\neq 0$ and $f\in S_{2}(N)$ is the unique newform.  Using the decomposition \eqref{decomplocalmaass} of $\mathcal{F}_{0,N,D_0,D}$, for every $M\in \Gamma_0(N)$ we have 
\begin{equation}\label{eqn:M-1}
0=\mathcal{F}_{0,N,D_0,D}\Big|_{0}(M-1) = \mathcal{P}_{D_0,D,\frac{3}{4}}\Big|_{0}(M-1) + a_D \left(\mathcal{E}_{f}\Big|_{0}(M-1) + f^*\Big|_0(M-1)\right).
\end{equation}
Hence if $\mathcal{E}_{f}(z)\Big|_{0}(M-1) + f^*(z)\Big|_0(M-1)\neq 0$, we obtain that $\mathcal{P}_{D_0,D,\frac{3}{4}}(z)\Big|_{0}(M-1)\neq 0$.  However, $\mathcal{E}_{f}(z)\Big|_{0}(M-1) + f^*(z)\Big|_0(M-1)$ is a constant independent of $z$. We choose $M=M_0$ to be a matrix sending $x_{N,1}$ to $x_{N,2}$ and verify that this difference is non-zero (this may be done by computing the periods of $f$). Note that this choice is made independent of $D$ and so the non-vanishing of $\mathcal{P}_{D_0,D,\frac{3}{4}}\Big|_{0}(M_0-1)$ for one $D$ implies that $\mathcal{E}_{f}\Big|_{0}(M_0-1) + f^*\Big|_0(M_0-1)\neq 0$.  One such choice of $D$ may be chosen from Table \ref{tab:discs} (e.g. $D=-11$ for $N=32$), which hence validates our choice of $M_0$.  Since \eqref{eqn:M-1} is independent of $z$, we then take the limit $z\to x\in\Q$ to show that \eqref{eqn:Px} is not $\Gamma_0(N)$-invariant.

We have hence concluded that $\f_{1,N,D_0,D}$ is in the space spanned by Eisenstein series if and only if \eqref{eqn:Px} is invariant under the action of $M_0$.  Now note that \eqref{eqn:Px} is a (non-zero) constant multiple of $F_{0,N,D,D_0}(x)$, defined in \eqref{eqn:F1}.  By Lemma \ref{lem:f1fchi}, the function $f_{1,N,D,D_0}$ is hence in the space spanned by Eisenstein series if and only if $F_{0,N,D,D_0}(x)$ is invariant under $M_0$.  

We finally use Lemma \ref{lem:Lval} to conclude the connection to central $L$-values.  However, we first have to show that an $m_0$ admissible for $D_0$ exists which satisfies the conditions in Lemma \ref{lem:Lval}.   

If the cuspidal part of every $f_{1,N,D,D_0}$ vanished, then $F_{0,N,D,D_0}$ would be $\Gamma_0(N)$-invariant for every $D$.  One easily checks (by a finite calculation) that this is not true for at least one choice, as given in Table \ref{tab:discs}. Moreover, picking one such choice for $m_0$, one easily checks that $L\left(f\otimes \chi_{m_0},1\right)\neq 0$.  

\end{proof}
\subsection{Connections to combinatorial questions and the proofs of Corollaries \ref{maincor} and \ref{cor:sumtwocubes}}

In this section, we recall relations between vanishing of central $L$-values and combinatorial questions. The first is the congruent number problem.
\begin{proof}[Proof of Corollary \ref{maincor}]
The first claim follows immediately from Theorem \ref{mainthm}. The second claim follows by combining Theorem \ref{mainthm} with the well-known fact (see the introduction in \cite{Tunnell} or Koblitz's book \cite{Koblitz} for a complete description) that $D$ is a congruent number if and only if the group of rational points of $E_D$ is infinite. As pointed out in \cite{Tunnell}, if $L(E_D,1)\neq 0$, then Coates--Wiles \cite{Coates-Wiles} implies that there are only finitely many rational points. The converse statement follows from BSD.
\end{proof}
We next consider the case of the elliptic curve $E_{27}$ given by $X^3+Y^3=1$.
\begin{proof}[Proof of Corollary \ref{cor:sumtwocubes}]
Applying Coates--Wiles \cite{Coates-Wiles} again, we obtain one direction of the claim. Assuming the Birch and Swinnerton-Dyer conjecture, there are infinitely many rational points on $E_{27,D}$ if and only if $L(E_{27,D},1)=0$. 

\end{proof}

\subsection{Proof of Corollary \ref{cor:SDodd}}
Here we conclude with the proof of Corollary \ref{cor:SDodd}.
\begin{proof}[Proof of  Corollary \ref{cor:SDodd}]
By Theorem \ref{mainthm}, it suffices to show that $F_{0,32,D,-3}(0)\neq F_{0,32,D,-3}\left(\frac{1}{3}\right)$.  Directly by definition \eqref{eqn:F1}, for $x\in\Q$ we have
\begin{equation}\label{eqn:Fcong}
F_{0,32,D,-3}(x)\equiv \#\left\{ Q=[a,b,c]\in \mathcal{Q}_{-3D}: \chi_{-3}(Q)\neq 0, 32|a, Q(x)>0>a\right\}\pmod{2}.
\end{equation}
However, since $D\not\equiv 0\pmod{3}$, we have that $\left(a,b,c,-3\right)=1$ and hence $\chi_{-3}(Q)\neq 0$ for every $Q\in  \mathcal{Q}_{-3D}$.  Noting that 
\[
Q\left(\frac{1}{3}\right) = \frac{a}{9}+ \frac{b}{3}+c,
\]
we have 
\[
F_{0,32,D,-3}\left(\frac{1}{3}\right)\equiv \# \mathcal{S}_D\pmod{2}.
\]

Hence, to prove the corollary it suffices to show that $F_{0,32,D,-3}(0)$ is even under the given assumptions on $D$.  For $x=0$, the condition on $Q=[a,b,c]\in \mathcal{Q}_{-3D}$ in \eqref{eqn:Fcong} is simply $c>0>a$ and $32|a$.  Since this condition is independent of $b$, for any $[a,b,c]$ satisfying the given conditions, we have that $[a,-b,c]$ is also in the set.  This matches pairs of quadratic forms whenever $b\neq 0$ and we obtain  
\begin{equation}\label{eqn:Fcong2}
F_{0,32,D,-3}(0)\equiv \#\left\{ Q=[a,0,c]\in \mathcal{Q}_{-3D}: 32|a, c>0>a\right\}\pmod{2}.
\end{equation}
However, since $D\equiv  -3\pmod{8}$, we have $-3D\equiv 1\pmod{8}$, while the discriminant of $[a,0,c]$ is $-4ac$, which is divisible by $4$.  The set on the right-hand side of \eqref{eqn:Fcong2} is hence empty, and we obtain the desired result.
\end{proof}


\begin{thebibliography}{99}
\bibitem{Alfes}C. Alfes, \begin{it}CM values and Fourier coefficients of harmonic Maass forms\end{it}, Ph.D. thesis, Technischen Universit\"at Darmstadt, 2015.
\bibitem{AE} C. Alfes, S. Ehlen, \begin{it}Twisted Traces of CM values of weak Maass forms\end{it}, J. Number Theory \textbf{133} (2013), 1827--1845.
\bibitem{AS} M. Abromowitz and I. Stegun, Handbook of mathematical functions, Dover Publications, New York, 1972, 1--1043.
\bibitem{BSDNotes} B. Birch and H. Swinnerton-Dyer, \begin{it}Notes on elliptic curves, II\end{it}, J. Reine Angew. Math. {\bf 218}, (1965), 79--108.
\bibitem{boautgra} R. E. Borcherds,\begin{it}Automorphic forms with singularities on {G}rassmannians\end{it}, Invent. Math., \textbf{132(3)} (1998), 491--562.
\bibitem{BKK} K. Bringmann, B. Kane, and W. Kohnen, \begin{it}Locally harmonic Maass forms and the kernel of the Shintani lift\end{it}, Int. Math. Res. Not. \textbf{2015} (2015), 3185--3224.
\bibitem{BKM} K. Bringmann, B. Kane, and M. Viazovska, \begin{it}Theta lifts and local Maass forms\end{it}, Math. Res. Lett. \textbf{20} (2013), 213--234.
\bibitem{BKS} K. Bringmann, B. Kane, and S. Zwegers, \begin{it}A generating function of locally harmonic Maass forms\end{it}, Compositio Math. \textbf{150} (2014), 749--762.
\bibitem{BringmannOno} K. Bringmann and K. Ono, {\it The $f(q)$ mock theta function conjecture and partition ranks}, Inventiones Mathematicae \textbf{165} (2006), 243--266. 
\bibitem{BruinierFunke} J. Bruinier and J. Funke, {\it On two geometric theta lifts,} Duke Math. Journal 125 (2004), 45-90. 
\bibitem{Coates-Wiles} J. Coates and A. Wiles, \begin{it} On the conjecture of Birch and Swinnerton-Dyer,\end{it} Invent. Math. \textbf{39} (1977), 223--251.
\bibitem{Dickson} L. Dickson, \begin{it}History of the Theory of Numbers. \end{it} Volume II, Chapter XVI, Dover Publications.
\bibitem{GrossKohnenZagier} B. Gross, W. Kohnen, and D. Zagier, \begin{it}Heegner points and derivatives of $L$-series, II\end{it}, Math. Ann. \textbf{497} (1987), 497--562.
\bibitem{Hoevel}M. H\"ovel, \begin{it}Automorphe Formen mit Singularit\"aten auf dem hyperbolichen Raum\end{it}, Ph.D. thesis, 2011.
\bibitem{Koblitz} N. Koblitz, Introduction to elliptic curves and modular forms, Springer, Berlin, 1993.
\bibitem{KohnenCoeff} W. Kohnen, \begin{it}Fourier coefficients of modular forms of half-integral weight,\end{it} Math. Ann. \begin{bf}271\end{bf} (1985), 237--268.
\bibitem{KohnenZagier}W. Kohnen and D. Zagier, \begin{it}Values of $L$-series of modular forms at the center of the critical strip\end{it} Invent. Math. \textbf{64} (1981), 175--198.
\bibitem{KongPhD}K. Kong, \begin{it}Modular local polynomial and the vanishing of $L$-values\end{it}, Ph.D. thesis, University of Hong Kong, 2017.
\bibitem{OnoCDM} K. Ono, \begin{it}Unearthing the visions of a master: harmonic Maass forms and number theory, \end{it} Proceedings of the 2008 Harvard-MIT Current Developments in Mathematics Conference, International Press, Somerville, MA, 2009, 347--454.
\bibitem{Parson} L. Parson, \begin{it} Modular integrals and indefinite binary quadratic forms\end{it} in ``A tribute to Emil Grosswald: number theory and related analysis,'' Contemp. Math \textbf{143} (1993), 513--523.
\bibitem{Shimura} G. Shimura, \begin{it} On modular forms of half integral weight\end{it}, Ann. Math. \textbf{97} (1973), 440--481.
\bibitem{Shintani} T. Shintani, \textit{On construction of holomorphic cusp forms of half integral weight}, Nagoya Math. J. \textbf{58} (1975), no. 1, 83--126.
\bibitem{Siegel} C. Siegel, Advanced analytic number theory, 2nd ed., Tata Institute of Fundamental Research Studies in Mathematics, 9, Tata Institute of Fundamental Research, Bombay, 1980.
\bibitem{Skoruppa} N. Skoruppa, \begin{it}Heegner cycles, modular forms and Jacobi forms\end{it}, J. Th\'eo. Nombres Bordeaux \textbf{3} (1991), 93--116.
\bibitem{StrombergWeilrep} F. Str\"{o}mberg. \begin{it}Weil representations associated with finite quadratic modules\end{it}, Math. Zeit. \textbf{275} (2013), 509--527.
\bibitem{Tunnell} J. Tunnell, \begin{it} A classical Diophantine problem and modular forms of weight $3/2$,\end{it} Invent. Math. \textbf{72} (1983), 323--334.
\bibitem{Vigneras} M. Vigneras, \begin{it}S\'eries theta des formes quadratiques ind\'efinies\end{it} in: Modular functions of one variable VI, Springer lecture notes \textbf{627} (1997), 227--239.
%\bibitem{VillegasZagier}F. Villegas-Rodriguez and D. Zagier, \begin{it}Which primes are the sum of two cubes?\end{it}, in Number Theory (Proceedings of the Third Conference of the Canadian Number Theory Association), CMS Conference proceedings \textbf{15} (1995), 295--306.
\bibitem{Waldspurger} J. Waldspurger, \begin{it}Sur les coefficients de Fourier des formes modulaires de poids demi-entier,\end{it} J. Math. Pures Appl. \textbf{60} (1981) 375--484.
\bibitem{ZagierRankin}D. Zagier, \begin{it}The Rankin--Selberg method for automorphic functions which are not of rapid decay\end{it}, J. Fac. Sci. Tokyo \textbf{28} (1982), 415--438.
\bibitem{ZagierQuadratic} D. Zagier, \begin{it} From quadratic functions to modular functions\end{it}, in Number Theory in Progress \textbf{2}, proceedings of international conference on number theory, Zakopane, 1997.
\bibitem{ZagierMellin}D. Zagier, \begin{it}The Mellin transform and other useful analytic techniques\end{it}, Appendix to E. Zeidler, \begin{it}Quantum field theory I: Basics in mathematics and physics.  A bridge between mathematicians and physicists\end{it}, Springer--Verlag, Berlin (2006), 305--323.
\end{thebibliography}
\end{document}